\documentclass[a4paper,11pt,twoside,reqno,final]{amsart}
\usepackage{amsfonts}
\usepackage{amsmath}
\usepackage{mathrsfs, enumerate, amssymb}
\usepackage{tikz}
\usepackage{fancybox}
\usepackage{pstricks}
\usepackage{pst-node}
\usepackage[headinclude,DIV13]{typearea}
\usepackage{hyperref}
\usepackage{color}
\usepackage{geometry}
\usepackage{graphicx, psfrag}

\newtheorem{theorem}{Theorem}

\newtheorem{corollary}{Corollary}
\newtheorem{definition}{Definition}

\newtheorem{lemma}{Lemma}

\newtheorem{remark}{Remark}
\renewenvironment{proof}[1][Proof]{\noindent\textbf{#1.} }{\ \rule{0.5em}{0.5em}}

\begin{document}
\title[Reduced two-type branching processes with infinite variance]{Reduced
two-type decomposable critical branching processes with possibly infinite
variance}
\thanks{This work is supported by the grant of RFBR 14-01-00318.}
\author{Charline Smadi}
\address{Irstea, UR LISC, Laboratoire d'Ing\'enierie des Syst\`emes
Complexes, 9 avenue Blaise Pascal-CS 20085, 63178 Aubi\`ere, France and
Department of Statistics, University of Oxford, 1 South Parks Road, Oxford
OX1 3TG, UK}
\email{charline.smadi@polytechnique.edu}
\author{Vladimir A. Vatutin}
\address{Department of Discrete Mathematics, Steklov Mathematical Institute,
8, Gubkin str., 119991, Moscow, Russia}
\email{vatutin@mi.ras.ru}
\date{\today }
\maketitle

\begin{abstract}
We consider a Galton-Watson process $\mathbf{Z}%
(n)=(Z_{1}(n),Z_{2}(n))$ with two types of particles. Particles of type 2
may produce offspring of both types while particles of type 1 may produce
particles of their own type only. Let $Z_{i}(m,n)$ be the number of
particles of type $i$ at time $m<n$ having offspring at time $n$. Assuming that the process is critical and
that the variance of the offspring distribution may be infinite
we describe the asymptotic behavior, as $m,n\rightarrow \infty $ of the law of $\mathbf{Z}(m,n)=(Z_{1}(m,n),Z_{2}(m,n))$ given $\mathbf{Z%
}(n)\neq \mathbf{0}$. We find three different types of
coexistence of particles of both types. Besides, we describe,
in the three cases, the distributions of the birth time and the type of the most recent common ancestor of individuals alive at time 
$n\rightarrow \infty .$
\end{abstract}

\section{Introduction}

Describing the genealogy of populations yields a better understanding of
patterns of variation in genetic data. It is a key issue in population
genetics and has been the subject of many investigations in recent years. In
this work, we are interested in the genealogy of a population modeled by a
two-type decomposable Galton-Watson branching process conditionally on its
survival until a large time $n$.

Let $Z_{i}(n)$ be the number of type $i$ particles $\left( i\in \left\{
1,2\right\} \right) $ alive at time $n\in \left\{ 0,1,...\right\} $, $%
\mathbf{Z}(n)=\left( Z_{1}(n),Z_{2}(n)\right) $, and denote by $\mathbf{Z}%
(0)=\mathbf{e}_{i}$ the event
\begin{equation*}
Z_{1}(0)+Z_{2}(0)=Z_{i}(0)=1.
\end{equation*}%
We assume that the process is critical, that is to
say
\begin{equation*}
\mathbf{E}\left[ Z_{i}(1)|\mathbf{Z}(0)=\mathbf{e}_{i}\right] =1,\quad i\in
\{1,2\}.
\end{equation*}%
Type $1$ particles may produce only particles of their own type while type $%
2$ particles are able to produce particles of both types. The production of
type $1$ particles by type $2$ particles can be seen as recurrent mutations
or migrations. In the finite variance case, this process has been studied in
\cite{sugitani1979limit, vatutin2014macroscopic}, and in the case of
infinite variance it has been investigated in \cite{Z} and \cite{VS}
in detail. Zubkov \cite{Z} and Vatutin and Sagitov \cite{VS}
have analyzed the asymptotic behavior of the survival probability of the
process given $\mathbf{Z}(0)=\mathbf{e}_{2}$, and described
the conditional limiting distribution of the population size at
time $n\rightarrow \infty $ given $\mathbf{Z}(n)\neq \mathbf{0}$
(where $\mathbf{0}=(0,0)$) and $\mathbf{Z}(0)=\mathbf{e}_{2}$. Our
motivation for considering the infinite variance case comes from \cite%
{eldon2006coalescent}. In this work Eldon and Wakeley, using genetic data
from a population of pacific oysters, showed that one individual was able to
generate enough offsprings in one birth event to produce a significant
proportion of the population. This is also the case for many
types of fungi and viruses.

The present paper deals with the structure of the family tree of the process
$\left\{ \mathbf{Z}(m),0\leq m\leq n\right\} $ given $\mathbf{Z}%
(n)\neq \mathbf{0}$. More precisely, we investigate the properties of the
process
\begin{equation*}
\mathbf{Z}\left( m,n\right) =(Z_{1}(m,n),Z_{2}(m,n))
\end{equation*}%
where $Z_{i}\left( m,n\right) $ is the number of type $i$ particles alive at
time $m<n$ and having descendants at time $n$. The
process $\mathbf{Z}\left( \cdot ,n\right) $ is called a reduced branching
process and can be thought of as the family tree relating the individuals
alive at time $n$. Two important characteristics of the reduced process are
the birth moment $\beta _{n}$ and the type $\mathcal{T}_{n}$ of the most
recent common ancestor (MRCA) of all individuals alive at time $n$ defined
as
\begin{equation}
\beta _{n}=\max \left\{ m<n:Z_{1}\left( m,n\right) +Z_{2}\left( m,n\right)
=1\right\} ,\quad \mathcal{T}_{n}=\{i\in \{1,2\}:Z_{i}\left(
\beta _{n},n\right) =1\}.  \label{defMRCA}
\end{equation}%
Limiting reduced trees have been extensively studied in the
literature. In the monotype and finite variance case, the distribution of $%
\beta _{n}$ already appears in Zubkov \cite{zubkov1976limiting}, and the
full structure of the reduced tree is derived by Fleischmann and
Siegmund-Schultze \cite{fleischmann1977structure}. Analogous results in the
stable case and/or in the more general setting of multitype indecomposable
branching processes with finite or infinite offspring variance can be found
in Vatutin \cite{vatutin1977limit,Vat14} and Yakymiv \cite%
{yakymiv1981reduced}. Finally, Duquesne and Le Gall \cite{duquesne2002random}
derived general properties of limiting reduced trees in the continuous
setting.

For $s_{1},s_{2}\in \lbrack 0,1]$, let
\begin{equation*}
F_{1}(s_{1})=\mathbf{E}\left[ s_{1}^{Z_{1}(1)}|\mathbf{Z}(0)=\mathbf{e}_{1}%
\right] \quad \text{and}\quad F_{21}(s_{1},s_{2})=\mathbf{E}\left[
s_{1}^{Z_{1}(1)}s_{2}^{Z_{2}(1)}|\mathbf{Z}(0)=\mathbf{e}_{2}\right]
\end{equation*}%
be the offspring generating functions for particles of both types. Our basic
assumptions on the offspring distributions are formulated in terms of the
generating functions $F_{1}(s_{1})$ and $F_{21}(s_{1},s_{2})$ as $%
s_{1},s_{2}\uparrow 1$:
\begin{equation}
F_{1}(s_{1})=s_{1}+\left( 1-s_{1}\right) ^{1+\alpha _{1}}L_{1}\left(
1-s_{1}\right) ,  \label{F1}
\end{equation}%
and
\begin{equation}
F_{21}(s_{1},s_{2})=s_{2}+\left( 1-s_{2}\right) ^{1+\alpha _{2}}L_{2}\left(
1-s_{2}\right) -(A_{21}-\rho \left( s_{1},s_{2}\right) )\left(
1-s_{1}\right) ,  \label{F2}
\end{equation}%
where $0<\alpha _{1},\alpha _{2}\leq 1,$ $L_{1}(x)$ and $L_{2}(x)$ are
slowly varying functions as $x\downarrow 0$,
\begin{equation}
A_{21}=\mathbf{E}\left[ Z_{1}(1)|\mathbf{Z}(0)=\mathbf{e}_{2}\right] \in
\left( 0,\infty \right) ,  \label{defA21}
\end{equation}%
and the function $\rho \left( s_{1},s_{2}\right) \rightarrow 0$ as $%
s_{1}s_{2}\uparrow 1$. The form of the probability generating function $%
F_{1} $ is natural if we aim at modeling a population where the offspring
distribution has an infinite variance. Indeed, \eqref{F1} holds if and only
if the distribution of $Z_{1}(1)$ is in the domain of attraction of a stable
law with index $1+\alpha _{1}$. Such a monotype branching process has been
investigated by many authors (see for example \cite%
{Sl68,borovkov1996distribution,wachtel2008limit,pakes2010critical}).

Introduce the following notation, for $i\in \{1,2\}$:
\begin{equation*}
Q_{i}(n)=\mathbf{P}\left( Z_{i}(n)>0|\mathbf{Z}(0)=\mathbf{e}_{i}\right)
\quad \text{and}\quad Q_{21}(n)=\mathbf{P}\left( \mathbf{Z}(n)\neq \mathbf{0}%
|\mathbf{Z}(0)=\mathbf{e}_{2}\right) .
\end{equation*}%
It is known \cite{Sl68} that, given Condition (\ref{F2}) with $s_{1}=1$ and
Condition (\ref{F1}),
\begin{equation}
Q_{i}(n)=(n+1)^{-1/\alpha _{i}}l_{i}(n),\quad n\rightarrow \infty ,
\label{AsQ}
\end{equation}%
where $l_{i}(n)$ is a slowly varying function as $n\rightarrow \infty ,$%
\begin{equation}
\alpha _{i}nQ_{i}^{\alpha _{i}}(n)L_{i}\left( Q_{i}(n)\right) \sim 1,\quad
n\rightarrow \infty ,  \label{ReprQ}
\end{equation}%
and for $\lambda >0$,
\begin{equation}
\lim_{n\rightarrow \infty }\mathbf{E}\Big[e^{-\lambda Q_{i}(n)Z_{i}(n)}|%
\mathbf{Z}(0)=\mathbf{e}_{i},Z_{i}(n)>0\Big]=1-(1+\lambda ^{-\alpha
_{i}})^{-1/\alpha _{i}}.  \label{monotype}
\end{equation}

We will see that three types of behavior emerge for the reduced trees,
depending on the relative values of $\alpha _{1}$ and $\alpha _{2}$. More
precisely, we will distinguish three cases:
\begin{equation}
nQ_{1}(n)=o\left( Q_{2}(n)\right) ,\ n\rightarrow \infty ,\quad \text{%
(Theorems \ref{T_ZubQ1} and \ref{T_ZubQ1micro})},  \label{Q1NeglQ2}
\end{equation}%
\begin{equation}
Q_{2}(n)=o(nQ_{1}(n)),\ n\rightarrow \infty ,\quad \text{(Theorems \ref%
{T_zubQ3} and \ref{T_zubQ3micro})},  \label{VaSag0}
\end{equation}%
\begin{equation}
nQ_{1}(n)\sim \sigma Q_{2}(n),\ \sigma \in \left( 0,\infty \right) ,\
n\rightarrow \infty ,\quad \text{(Theorems \ref{T_VatSag} and \ref%
{T_VatSagmicro})}.  \label{VaSag}
\end{equation}%
We will show that under Condition~(\ref{VaSag0}), the reduced
processes have similar properties as in the case of finite variance studied
in detail by Vatutin in \cite{Vat14} and \cite{VD15}, while 
under condition~(\ref{Q1NeglQ2}) or~(\ref{VaSag}) the dynamics
of the reduced processes has an essentially different nature.\newline

The paper is organized as follows. In Section \ref{sectionmacro}
we describe the local behavior of the reduced process. Section \ref%
{sectionmicro} is devoted to its global behavior and to the law of the MRCA
of individuals alive at a large time $n$. In Section \ref%
{sectionauxresults} we give auxiliary results needed in the proofs. Sections %
\ref{proofTh1} to \ref{proofTh5} are devoted to the analysis of the local
structure of the reduced process. In Sections \ref%
{S_finiteDimen} and \ref{S_tightness} we consider the global behavior of the
reduced process. Finally, in Section \ref{S_MRCA}, we
derive the MRCA limit distributions.\newline 

In the sequel $\mathbf{P}_{i}$ and $\mathbf{E}_{i}$ will denote
the probability and expectation conditionally on $\mathbf{Z}(0)=\mathbf{e}%
_{i}$. Sometimes it will be convenient to write $\mathbf{P}$ and $\mathbf{E}$
for $\mathbf{P}_{2}$ and $\mathbf{E}_{2}$. Finally, the relations
\begin{equation*}
a(n)\sim b(n),\quad a(n)=O(b(n)),\quad a(n)=o(b(n)),\quad \text{and}\quad
a(n)\ll b(n)
\end{equation*}%
will be understood (if otherwise is not stated) as $n\rightarrow \infty $
and the symbols $C,C_{1},C_{2},....$ will denote positive constants which may vary from line to line.

\section{Local characteristics of the limit processes}

\label{sectionmacro}

We first deal with Condition \eqref{Q1NeglQ2}. Notice that $%
\alpha _{1}$ is necessarily less than $1$, as otherwise $nQ_{1}(n)$ would be
a slowly varying function while $Q_{2}(n)$ is a regularly
varying function with negative index. This allows us to introduce an
integer-valued function $h^{\ast }(n)$ satisfying
\begin{equation}
\frac{\alpha _{1}A_{21}}{1-\alpha _{1}}h^{\ast }(n)Q_{1}(h^{\ast }(n))\sim
Q_{2}(n).  \label{Mu1}
\end{equation}%
Using properties of regularly varying functions (see, for instance, \cite%
{Sen77}, Ch.1, Lemma 1.8), we get
\begin{equation}
h^{\ast }(n)=n^{\alpha _{1}/(\alpha _{2}(1-\alpha _{1}))}l_{h^{\ast }}(n),
\label{hstar}
\end{equation}%
where $l_{h^{\ast }}(n)$ is a function slowly varying at infinity. Moreover,
$h^{\ast }(n)\ll n$ (see Lemma~\ref{L_regular} page~\pageref{Lemma1}). In what follows equalities of the form $\mathbf{Z}(x,n)=y$ for $%
x,y\in (0,\infty )$ will be understood as $\mathbf{Z}(\left[ x%
\right] ,n)=\left[ y\right] $ where $\left[ w\right] $ stands for the
integer part of $w$. The same agreement will be in force in other similar
situations. For instance, $Q_{i}(x),x\in \lbrack 0,\infty )$ will be treated
as $Q_{i}(\left[ x\right] )$.

\begin{theorem}
\label{T_ZubQ1}Let Conditions \eqref{F1}-\eqref{defA21} and \eqref{Q1NeglQ2}
be satisfied and $(s_{1},s_{2})$ be in $(0,1)^2$.

\begin{enumerate}
\item[(0)] If $1\leq m\ll n$, then
\begin{equation}
\lim_{n\rightarrow \infty }\mathbf{E}\left[
s_{1}^{Z_{1}(m,n)}s_{2}^{Z_{2}(m,n)}|\mathbf{Z}(n)\neq \mathbf{0}\right]
=s_{2};  \label{Zub11}
\end{equation}

\item[(1)] For any $a\in (0,1)$
\begin{equation}
\lim_{n\rightarrow \infty }\mathbf{E}\left[
s_{1}^{Z_{1}(an,n)}s_{2}^{Z_{2}(an,n)}|\mathbf{Z}(n)\neq \mathbf{0}\right]
=1-\Big(a+\left( 1-a\right) (1-s_{2})^{-\alpha _{2}}\Big)^{-\frac{1}{\alpha
_{2}}};  \label{Zub12}
\end{equation}

\item[(2)] For any function $h$ satisfying $h^{\ast }(n)\ll h(n)\ll n$ and $%
m=n-h(n)$, and $\lambda _{2}\geq 0$
\begin{equation*}
\lim_{n\rightarrow \infty }\mathbf{E}\left[ s_{1}^{Z_{1}(m,n)}\exp \left\{
-\lambda _{2}\frac{Q_{2}(n)}{Q_{21}(h(n))}Z_{2}(m,n)\right\} |\mathbf{Z}%
(n)\neq \mathbf{0}\right] =1-\left( 1+\lambda _{2}^{-\alpha _{2}}\right)
^{-1/\alpha _{2}};
\end{equation*}

\item[(3)] For any $t>0$, $\lambda _{2}\geq 0$ and $m=n-th^{\ast }(n)$
\begin{multline*}
\lim_{n\rightarrow \infty }\mathbf{E}\left[ s_{1}^{Z_{1}(m,n)}\exp \left\{
-\lambda _{2}\frac{Q_{2}(n)}{Q_{21}(xh^{\ast }(n))}Z_{2}(m,n)\right\} |%
\mathbf{Z}(n)\neq \mathbf{0}\right]  \\
=1-\left( 1+\left( t^{1-1/\alpha _{1}}(1-s_{1})^{1-\alpha _{1}}+\lambda
_{2}\right) ^{-\alpha _{2}}\right) ^{-1/\alpha _{2}};
\end{multline*}

\item[(4)] For any function $h$ satisfying $1\ll h(n)\ll h^{\ast }(n)$, $%
m=n-h(n)$, and $\lambda _{1},\lambda _{2}\geq 0$
\begin{multline*}
\lim_{n\rightarrow \infty }\mathbf{E}\left[ \exp \left\{ -\lambda _{1}\frac{%
Q_{1}(h^{\ast }(n))}{Q_{1}(h(n))}Z_{1}(m,n)-\lambda _{2}\frac{Q_{2}(n)}{%
Q_{21}(h(n))}Z_{2}(m,n)\right\} \Big|\mathbf{Z}(n)\neq \mathbf{0}\right]  \\
=1-\left( 1+\left( \lambda _{1}^{1-\alpha _{1}}+\lambda _{2}\right)
^{-\alpha _{2}}\right) ^{-1/\alpha _{2}}.
\end{multline*}
\end{enumerate}
\end{theorem}

We see that under the conditions of Theorem \ref{T_ZubQ1} the reduced
process essentially consists of one particle of the second type at the
initial stage, it consists of many second type particles but still no
particles of the first type at the intermediate stage, and particles of the
first type appear in the process only at time $m=n-th^{\ast }(n)$, $t>0$. As a consequence, the coexistence of both types in the
reduced process is possible only within a relatively small (in comparison
with $n$) time interval of order $h^{\ast }(n)$, located at the end of
the time frame under consideration,~$[0,n].$ Thus, the evolution of this
reduced process is mainly supported by particles of the initial type.\newline

Let us now deal with Condition \eqref{VaSag0}. To formulate the desired
result we need to introduce a function $\phi =\phi (\lambda _{1},\lambda
_{2}),\lambda _{1},\lambda _{2}\geq 0,$ which solves the partial
differential equation
\begin{equation}
(1+\alpha _{2})\lambda _{1}\frac{\partial \phi }{\partial \lambda _{1}}%
+\lambda _{2}\frac{\partial \phi }{\partial \lambda _{2}}=\phi -\phi
^{1+\alpha _{2}}+\alpha _{2}A_{21}\lambda _{1}  \label{defphi}
\end{equation}%
with initial conditions
\begin{equation*}
\phi \left( 0,0\right) =0,\quad \frac{\partial \phi \left( \lambda
_{1},\lambda _{2}\right) }{\partial \lambda _{1}}\left\vert _{\lambda
_{1}=\lambda _{2}=0}\right. =A_{21},\quad \frac{\partial \phi \left( \lambda
_{1},\lambda _{2}\right) }{\partial \lambda _{2}}\left\vert _{\lambda
_{1}=\lambda _{2}=0}\right. =1.
\end{equation*}
The problem of existence and uniqueness of a solution for Equation %
\eqref{defphi} has been solved in \cite{VS} and \cite{Sag87}. Note that if $%
\alpha _{1}=\alpha _{2}=1$, the solution of \eqref{defphi} has an explicit
form:
\begin{equation*}
\phi \left( \lambda _{1},\lambda _{2}\right) =\sqrt{A_{21}\lambda _{1}}\,%
\frac{\lambda _{2}+\sqrt{A_{21}\lambda _{1}}\tanh \sqrt{A_{21}\lambda _{1}}}{%
\lambda _{2}\tanh \sqrt{A_{21}\lambda _{1}}+\sqrt{A_{21}\lambda _{1}}}.
\end{equation*}
Let $g^{\ast }(n)$ be an integer-valued function satisfying
\begin{equation}
Q_{2}(g^{\ast }(n))\sim Q_{21}(n).  \label{defgast}
\end{equation}%
Using properties of regularly varying functions and the asymptotic
representation for $Q_{21}(n)$ given later on in \eqref{SmallQ2} one may
check that
\begin{equation}
g^{\ast }(n)=n^{\alpha _{2}/(\alpha _{1}(1+\alpha _{2}))}l_{g^\ast}(n),
\label{defgast2}
\end{equation}%
where $l_{g^\ast}(n)$ is a function slowly varying at infinity, 
and that $g^{\ast }(n)\ll n$.

\begin{theorem}
\label{T_zubQ3} Let Conditions (\ref{F1})-(\ref{defA21}) and (\ref{VaSag0})
be satisfied and $(s_{1},s_{2})$ be in $(0,1)^2$.

\begin{enumerate}
\item[(0)] If $1\leq m\ll g^{\ast }(n),$ then
\begin{equation*}
\lim_{n\rightarrow \infty }\mathbf{E}\left[
s_{1}^{Z_{1}(m,n)}s_{2}^{Z_{2}(m,n)}|\mathbf{Z}(n)\neq \mathbf{0}\right]
=s_{2};
\end{equation*}

\item[(1)] For any $t>0$
\begin{equation*}
\lim_{n\rightarrow \infty }\mathbf{E}\left[ s_{1}^{Z_{1}(tg^{\ast
}(n),n)}s_{2}^{Z_{2}(tg^{\ast }(n),n)}|\mathbf{Z}(n)\neq \mathbf{0}\right]
=1-t^{-1/\alpha _{2}}\phi \Big((1-s_{1})\frac{t^{1/\alpha _{2}+1}}{\alpha
_{2}A_{21}},(1-s_{2})t^{1/\alpha _{2}}\Big);
\end{equation*}

\item[(2)] If $g^{\ast }(n)\ll m\ll n$, then
\begin{equation*}
\lim_{n\rightarrow \infty }\mathbf{E}\left[
s_{1}^{Z_{1}(m,n)}s_{2}^{Z_{2}(m,n)}|\mathbf{Z}(n)\neq \mathbf{0}\right]
=1-(1-s_{1})^{1/(1+\alpha _{2})};
\end{equation*}

\item[(3)] For any $a\in (0,1)$
\begin{equation}
\lim_{n\rightarrow \infty }\mathbf{E}\left[
s_{1}^{Z_{1}(an,n)}s_{2}^{Z_{2}(an,n)}|\mathbf{Z}(n)\neq \mathbf{0}\right]
=1-\left( a+(1-a)(1-s_{1})^{-\alpha _{1}}\right) ^{-1/(\alpha _{1}(1+\alpha
_{2}))};  \label{T_zubQ3point4}
\end{equation}

\item[(4)] If $m=n-h(n)$ for some $1\ll h(n)\ll n$, then for $%
\mathbf{\lambda }_{1}\geq 0$
\begin{equation*}
\lim_{n\rightarrow \infty }\mathbf{E}\left[ \exp \left\{ -\lambda _{1}\frac{%
Q_{1}(n)}{Q_{1}(h(n))}Z_{1}(m,n)\right\} s_{2}^{Z_{2}(m,n)}|\mathbf{Z}%
(n)\neq \mathbf{0}\right] =1-\left( 1+\lambda _{1}^{-\alpha _{1}}\right)
^{-1/\left( \alpha _{1}(1+\alpha _{2})\right) }.
\end{equation*}
\end{enumerate}
\end{theorem}

Once again, coexistence of both types in the reduced process is possible
only within a relatively small time interval (here of order $g^{\ast }(n)$). 
But unlike the case of Theorem \ref{T_ZubQ1} this interval is located at
the beginning of the population evolution and the reduced
process is essentially composed of type $1$ individuals.\newline

The next theorem deals with Condition \eqref{VaSag}. Let $b\in
\left( 1,\infty \right) $ be the unique positive solution of
\begin{equation}
x^{1+\alpha _{2}}-x=\sigma \alpha _{2}A_{21},  \label{defb}
\end{equation}%
and $\psi =\psi \left( \lambda _{1},\lambda _{2}\right) ,\lambda
_{1},\lambda _{2}\geq 0,$ the unique solution of the quasi-linear partial
differential equation
\begin{equation}
\left( 1+\alpha _{2}\right) \lambda _{1}\frac{\partial \psi }{\partial
\lambda _{1}}+\lambda _{2}\frac{\partial \psi }{\partial \lambda _{2}}=\psi
-b^{\alpha _{2}}\psi ^{1+\alpha _{2}}+\alpha _{2}A_{21}\lambda _{1}\left(
1+\left( \frac{b}{\sigma }\lambda _{1}\right) ^{\alpha _{1}}\right) ^{-\frac{%
1}{\alpha _{1}}}  \label{Partial2}
\end{equation}%
with initial conditions
\begin{equation}
\psi \left( 0,0\right) =0,\quad \frac{\partial \psi \left( \lambda
_{1},\lambda _{2}\right) }{\partial \lambda _{1}}\left\vert _{\lambda
_{1}=\lambda _{2}=0}\right. =A_{21},\quad \frac{\partial \psi \left( \lambda
_{1},\lambda _{2}\right) }{\partial \lambda _{2}}\left\vert _{\lambda
_{1}=\lambda _{2}=0}\right. =1.  \label{Inpsi}
\end{equation}

Then we have the following asymptotics for the intermediate case %
\eqref{VaSag}:

\begin{theorem}
\label{T_VatSag} Let Conditions (\ref{F1})-(\ref{defA21}) and (\ref{VaSag})
be satisfied, and  $(s_{1},s_{2})$ be in $(0,1)^{2}$.

\begin{enumerate}
\item[(0)] If $1\leq m\ll n$, then
\begin{equation}
\lim_{n\rightarrow \infty }\mathbf{E}\left[
s_{1}^{Z_{1}(m,n)}s_{2}^{Z_{2}(m,n)}|\mathbf{Z}(n)\neq \mathbf{0}\right]
=s_{2};  \label{Negl2}
\end{equation}

\item[(1)] For any $a\in (0,1)$,
\begin{multline*}
\lim_{n\rightarrow \infty }\mathbf{E}\left[
s_{1}^{Z_{1}(an,n)}s_{2}^{Z_{2}(an,n)}\Big|\mathbf{Z}(n)\neq \mathbf{0}%
\right] \\
=1-a^{-1/\alpha _{2}}\psi \Big((1-s_{1})\frac{\sigma }{b}\Big(%
\frac{a}{1-a}\Big)^{1/\alpha _{1}},(1-s_{2})\Big(\frac{a}{1-a}\Big)%
^{1/\alpha _{2}}\Big);
\end{multline*}

\item[(2)] If $m=n-h(n)$, where $1\ll h(n) \ll n $, then for $\lambda_{1},\lambda _{2}\geq 0$
\begin{equation*}
\lim_{n\rightarrow \infty }\mathbf{E}\left[ \exp \left\{ -\frac{\lambda _{1}
Q_{21}(n)}{nQ_{1}(h(n))}Z_{1}(m,n)-\frac{\lambda _{2}Q_{2}(n)}{Q_{2}(h(n))}%
Z_{2}(m,n)\right\} \Big|\mathbf{Z}(n)\neq \mathbf{0}\right] =1-\psi \left(
\lambda _{1},\lambda _{2}\right) .
\end{equation*}
\end{enumerate}
\end{theorem}

In this last case, both types of particles coexist during an interval of
order $n$.

As we will see in Section \ref{sectionmicro}, Theorems \ref{T_ZubQ1}-\ref{T_VatSag} will be elements of the proofs of more general
statements dealing with conditional functional limit theorems for reduced
branching processes.

\section{Global description of the limit processes}

\label{sectionmicro}

In this section, we study in detail the dynamics of the reduced process, in
particular, the transitions between monotype and two-type phases. 
Let $x^{\left[ k\right] }=x(x-1)\cdot \cdot \cdot (x-k+1)$.
In Definitions \ref{defprocessX}-\ref{defprocessG}, we specify five auxiliary continuous time Markov
processes needed to
describe the transitions.

\begin{definition}
\label{defprocessX} We denote by $\{\mathbf{\nu }=(\nu _{1},\nu
_{2}),(g_{1}(s_{1}),\mu _{1}),(g_{21}(s_{1},s_{2}),\mu _{2})\}$ a
homogeneous Markov branching process initiated at time $t=0$ by $\nu _{1}$
particles of type $1$ and  $\nu _{2}$ particles of type 
$2$. At rate $\mu _{1}$ ${(}${resp. $\mu _{2}$}${)}$ a
particle of type $1$ ${(}${resp. $2$}${)}$ dies and produces children in
accordance with the probability generating function $g_{1}$ ${(}${resp. $%
g_{21}$}${)}$. The newborn particles behave as their parents and
independently of each other and of the past. Notice that when type $1$ ${(}$%
{resp. $2$}${)}$ particles do not evolve (no death, no offspring), $%
g_{1}(s_{1})=s_{1}$ ${(}${resp. $g_{21}(s_{1},s_{2})=s_{2}$}${)}$ and we
can take any $\mu _{1}$ ${(}${resp. $\mu _{2}$}${)}$. We will write $\mu
_{i}$, $i\in \{1,2\}$ to make it clear in the definition. With this
representation in hands we can now introduce the following four
branching processes: 

\begin{enumerate}
\item $\left\{ X(t),t\geq 0\right\}= \{ \mathbf{e}_2, (s_1,\mu_1),
(g^{(X)}_{2}(s_1,s_2),1) \} $, where
\begin{equation}
g_{2}^{(X)}(s_1,s_{2})=\frac{1}{\alpha _{2}}((1-s_{2})^{1+\alpha
_{2}}-1+(1+\alpha _{2})s_{2})=\frac{1}{\alpha _{2}}\sum_{k=2}^{\infty }\frac{%
(1+\alpha _{2})^{[k]}}{k!}(-1)^{k}s_{2}^{k}.  \label{Defg2}
\end{equation}

\item $\left\{ \mathbf{Y}(t)=\left( Y_{1}(t),Y_{2}(t)\right) ,t\geq
0\right\}= \{ \mathbf{e}_2, (s_1,\mu_1),
(g^{(Y)}_{21}(s_1,s_2),1+1/\alpha_2) \} $, where
\begin{eqnarray}
g_{21}^{(Y)}(s_{1},s_{2}) &=&\frac{1}{1+\alpha _{2}}\Big((1-s_{2})^{1+\alpha
_{2}}-1+(1+\alpha _{2})s_{2}+s_{1}\Big)  \notag \\
&=&\frac{1}{1+\alpha _{2}}\sum_{k=2}^{\infty }\frac{\left( 1+\alpha
_{2}\right) ^{[k]}}{k!}(-1)^{k}s_{2}^{k}+\frac{1}{1+\alpha _{2}}s_{1}.
\label{pgfeq1}
\end{eqnarray}

\item $\left\{ V(t),t\geq 0\right\} =\{\mathbf{e}%
_{1},(g_{1}^{(V)}(s_{1}),1),(s_{2},\mu _{2})\}$, where
\begin{equation} \label{pgfeq4}
g_{1}^{(V)}(s_{1})=\frac{1}{\gamma _{1}-1}\Big[(1-s_{1})^{\gamma _{1}}+\gamma _{1}s_{1}-1\Big]=%
\frac{1}{\gamma_1-1}\sum_{k=2}^{\infty }\frac{\gamma_1^{[k]}}{k!}(-1)^{k}s_{1}^{k},
\end{equation}%
and $\gamma _{1}=1+\alpha _{1}(1+\alpha _{2})$.

\item $\left\{ \mathbf{W}(t)=\left( W_{1}(t),W_{2}(t)\right) ,t\geq
0\right\} =\{\mathbf{e}%
_{2},(g_{1}^{(W)}(s_{1}),1),(g_{21}^{(W)}(s_{1},s_{2}),\kappa )\}$, where
\begin{equation}
\kappa :\frac{(1+\alpha _{2})b^{\alpha _{2}}-1}{\alpha _{2}}>1,
\label{defkappa}
\end{equation}%
\begin{equation}
g_{1}^{(W)}(s_{1})=\frac{1}{\alpha _{1}}\Big((1-s_{1})^{1+\alpha
_{1}}-1+(1+\alpha _{1})s_{1}\Big)=\frac{1}{\alpha _{1}}\sum_{k=2}^{\infty }%
\frac{(1+\alpha _{1})^{[k]}}{k!}(-1)^{k}s_{1}^{k},  \label{pgfeq3}
\end{equation}%
and
\begin{eqnarray}
g_{21}^{(W)}(s_{1},s_{2}) &=&\frac{1}{\alpha_2\kappa }\Big(%
\frac{\sigma \alpha _{2}A_{21}}{b}s_{1}+b^{\alpha_2}\Big(%
(1-s_{2})^{1+\alpha _{2}}-1+(1+\alpha _{2})s_{2}\Big)\Big)  \notag \\
&=&\frac{\sigma A_{21}}{b\kappa}s_{1}+\frac{b^{\alpha _{2}}}{%
\alpha _{2}\kappa}\sum_{k=2}^{\infty }\frac{(1+\alpha
_{2})^{[k]}}{k!}(-1)^{k}s_{2}^{k}.  \label{pgfeqW}
\end{eqnarray}
\end{enumerate}
\end{definition}

Note that $X(\cdot )$ and $V(\cdot )$
are monotype processes, while $\mathbf{Y}(\cdot )$ and $\mathbf{W}(\cdot )$
are two-type processes. Moreover, particles of type $1$ of the
process $\mathbf{Y}(\cdot )$ are sterile and immortal.

\begin{definition}
\label{defprocessG} We denote by $\{\mathbf{G}%
(t)=(G_{1}(t),G_{2}(t)),t\geq 0\}$ a two-dimensional Markov process with
values in $\mathbf{N}_{0}\times $ $\mathbf{R}_{+}=\left\{ 0,1,2,...\right\}
\times (0,\infty )$. The initial state of $\mathbf{G}(\cdot )$ is random:%
\begin{equation*}
\mathbf{G}(0)=\left( G_{1}(0),G_{2}(0)\right) =\left( 0,\theta _{2}\right)
\end{equation*}%
where the distribution of $\theta _{2}$ is given by
\begin{equation}
\mathbf{E}\left[ e^{-\lambda \theta _{2}}\right] =1-\left( 1+\lambda
^{-\alpha _{2}}\right) ^{-1/\alpha _{2}},\quad \lambda >0.  \label{lawtheta}
\end{equation}%
The distribution of $\mathbf{G}(t)$ at time $t>0$ is specified by
\begin{equation*}
\mathbf{E}\left[ s_{1}^{G_{1}(t)}e^{-\lambda _{2}G_{2}(t)}\right] =1-\left(
1+\left( t^{1/\alpha _{1}-1}(1-s_{1})^{1-\alpha _{1}}+\lambda _{2}\right)
^{-\alpha _{2}}\right) ^{-1/\alpha _{2}},0\leq s_{1}\leq 1 ,\lambda
_{2}>0,
\end{equation*}%
and the transition probabilities for $0\leq t_{0}<t_{1}<\infty $ and $\left(
n_{1},y\right) \in \mathbf{N}_{0}\times $ $\mathbf{R}_{+}$ are given by
\begin{multline*}
\mathbf{E}\left[ s_{1}^{G_{1}(t_{1})}e^{-\lambda
_{2}G_{2}(t_{1})}|\,G_{1}(t_{0})=n_{1},G_{2}(t_{0})=y\right] =\left( 1-\left[
1-\frac{t_{1}}{t_{0}}+\frac{t_{1}}{t_{0}}(1-s_{1})^{-\alpha
_{1}}\right] ^{-1/\alpha _{1}}\right) ^{n_{1}} \\
\times \exp \left\{ -\left( t_{1}^{1-1/\alpha _{1}}(1-s_{1})^{1-\alpha
_{1}}\left( 1-\left( 1+\left( \frac{t_{0}}{t_{1}}-1\right) (1-s_{1})^{\alpha
_{1}}\right) ^{1-1/\alpha _{1}}\right) +\lambda _{2}\right) y\right\} .
\end{multline*}
\end{definition}

In the sequel the symbol $\mathcal{L}_{\mu }$ will denote the law with
initial condition $\mu $ which may be a measure or a random variable, and we will write 
$\mathcal{L}^{(n)}\left( \cdot
\right) $ for $\mathcal{L}(\cdot |\mathbf{Z}(n)\neq
\mathbf{0})$. Moreover, we say that a process $\{\mathbf{\Theta }%
(t),t\geq 0\}$ has a.s. constant paths with (random) value $\mathbf{\Theta }$
if
\begin{equation*}
\mathbf{P}(\mathbf{\Theta }(t)=\mathbf{\Theta }\text{ for all }t\in
(0,\infty ))=1.
\end{equation*}%

We have the following global description of the reduced process 
under Condition \eqref{Q1NeglQ2} and the assumption $\mathbf{Z}%
(0)=\mathbf{e}_{2},$ where the symbol $\Longrightarrow $ means convergence
in the respective space of c\`adl\`ag functions endowed with
Skorokhod topology.

\begin{theorem}
\label{T_ZubQ1micro}Let Conditions \eqref{F1}-\eqref{defA21} and %
\eqref{Q1NeglQ2} be satisfied.

\begin{enumerate}
\item If $m=\left[ (1-e^{-t})n\right] ,\,t\geq 0,$ then%
\begin{equation*}
\mathcal{L}^{(n)}\{\mathbf{Z}(m,n),0\leq t<\infty \}\Longrightarrow \mathcal{%
L}_{(0,1)}\{(0,X(t)),0\leq t<\infty \};
\end{equation*}

\item If $m=n-th(n)$, $t>0,$ for $h^{\ast }(n)\ll h(n)\ll n,$ then
\begin{equation*}
\mathcal{L}^{(n)}\left\{ \left( Z_{1}(m,n),\frac{Q_{2}(n)}{Q_{21}(h(n))}%
Z_{2}(m,n)\right) ,\,0<t<\infty \right\} \Longrightarrow \mathcal{L}%
\{(0,\theta _{2}(t)),0<t<\infty \},
\end{equation*}%
where the limiting process has a.s. constant paths with value $%
\theta _{2}$ specified by (\ref{lawtheta});

\item If $m=n-th^{\ast }(n)$, $t>0,$ then%
\begin{equation*}
\mathcal{L}^{(n)}\left\{ \left( Z_{1}(m,n),\frac{Q_{2}(n)}{Q_{21}(th^{\ast
}(n))}Z_{2}(m,n)\right) ,\,0<t<\infty \right\} \Longrightarrow \mathcal{L}\{\mathbf{G}(t),0<t<\infty \};
\end{equation*}

\item For any $\lambda _{1},\lambda _{2}>0,$ any function $h(n)$ satisfying $%
1\ll h(n)\ll h^{\ast }(n)$ and $m=n-th(n)$
\begin{equation*}
\mathcal{L}^{(n)}\left\{ \left( \frac{Q_{1}(h^{\ast }(n))}{Q_{1}(th(n))}%
Z_{1}(m,n),\frac{Q_{2}(n)}{Q_{21}(th(n))}Z_{2}(m,n)\right) ,0<t<\infty
\right\} \Longrightarrow \mathcal{L}\{\mathbf{\Theta} (t),0<t<\infty \},
\end{equation*}%
where the limiting process has a.s. constant paths with value $%
\left( \theta _{1},\theta _{2}\right) $ specified by
\begin{equation*}
\mathbf{E}\Big[e^{-\lambda _{1}\theta _{1}}e^{-\lambda _{2}\theta _{2}}\Big]%
=1-\left( 1+\left( \lambda _{1}^{1-\alpha _{1}}+\lambda _{2}\right)
^{-\alpha _{2}}\right) ^{-1/\alpha _{2}}.
\end{equation*}
\end{enumerate}
\end{theorem}

Let $d_{n}$ and $q_{n},$ $n=1,2,...,$ be positive functions such that $%
d_{n}q_{n}=1$ and
\begin{equation*}
\lim_{n\rightarrow \infty }d_{n}=\infty \text{ and }d_{n}n^{-\varepsilon }=0%
\text{ for any }\varepsilon >0\text{.}
\end{equation*}%
Then we have the following result under Condition \eqref{VaSag0}:

\begin{theorem}
\label{T_zubQ3micro} Let Conditions (\ref{F1})-(\ref{defA21}) and (\ref%
{VaSag0}) be satisfied. Then

\begin{enumerate}
\item
\begin{equation}
\mathcal{L}^{(n)}\{\mathbf{Z}(tg^{\ast }(n),n),0\leq t<\infty
\}\Longrightarrow \mathcal{L}_{(0,1)}\{\mathbf{Y}(t),0\leq t<\infty \};
\label{Ylim}
\end{equation}%

\item If $\gamma _2=\alpha _{2}/(\alpha _{1}(1+\alpha _{2})),$ then%
\begin{equation*}
\mathcal{L}^{(n)}\{\mathbf{Z}(d_{n}g^{\ast }(n)n^{y},n),0\leq y<1-\gamma
_2\}\Longrightarrow \mathcal{L}_{(\theta _{1},0)}\left\{ (\theta
_{1}(y),0),0\leq y<1-\gamma _2\right\} ,
\end{equation*}%
where $\theta _{1}(\cdot )$ has a.s. constant paths with value $%
\theta _{1}$ specified by
\begin{equation}
\mathbf{E}\left[ s^{\theta _{1}}\right] =1-(1-s)^{1/(1+\alpha _{2})},\quad
0\leq s\leq 1.  \label{lawzeta0}
\end{equation}

\item If $m=\left[ \left( (1-e^{-t})+q_{n}\right) n\right] ,$ then%
\begin{equation*}
\mathcal{L}^{(n)}\{\mathbf{Z}(m,n),0\leq t<\infty \,\}\Longrightarrow
\mathcal{L}_{(\theta _{1},0)}\{(V_{1}(t),0),0\leq t<\infty \},
\end{equation*}%
where the law of $\theta _{1}$ has been specified in \eqref{lawzeta0}.

\item If $1\ll h(n)\ll n$ and $m=n-th(n)$ then
\begin{equation*}
\mathcal{L}^{(n)}\left\{ \left( \frac{Q_{1}(n)}{Q_{1}(th(n))}%
Z_{1}(m,n),Z_{2}(m,n)\right) ,0<t<\infty \right\} \Longrightarrow \mathcal{L}%
\{(\rho (t),0),0<t<\infty \},
\end{equation*}%
where the limiting process has a.s. constant paths whose value $%
\rho $ is specified by
\begin{equation*}
\mathbf{E}\left[ e^{-\lambda \rho }\right] =1-(1+\lambda ^{-\alpha
_{1}})^{-1/(\alpha _{1}(1+\alpha _{2}))}.
\end{equation*}
\end{enumerate}
\end{theorem}

For the case $\alpha _{1}=\alpha _{2}=1,$ where $%
2g_{21}^{(Y)}(s_{1},s_{2})=s_{2}^{2}+s_{1}$, this result is established in \cite{Vat14}.

Finally, we have the following result under Condition %
\eqref{VaSag} and the assumption $\mathbf{Z}(0)=\mathbf{e}_{2}$:

\begin{theorem}
\label{T_VatSagmicro}Let Conditions (\ref{F1})-(\ref{defA21}) and (\ref%
{VaSag}) be satisfied.

\begin{enumerate}
\item If $m=\left[ (1-e^{-t})n\right] $, then
\begin{equation*}
\mathcal{L}^{(n)}\{\mathbf{Z}(m,n),0\leq t<\infty \,\}\Longrightarrow
\mathcal{L}_{(0,1)}\{\mathbf{W}(t),0\leq t<\infty \};
\end{equation*}

\item If $n-m=th(n)$, where $1\ll h(n)\ll n$, then
\begin{equation*}
\mathcal{L}^{(n)}\left\{ \left( \frac{Q_{21}(n)}{nQ_{1}(th(n))}Z_{1}(m,n),%
\frac{Q_{2}(n)}{Q_{2}(th(n))}Z_{2}(m,n)\right) ,0<t<\infty \right\} \hspace{%
0cm}\Longrightarrow \mathcal{L}\{\mathbf{\Upsilon }(t),0<t<\infty \},
\end{equation*}%
where the limiting process has a.s. constant paths 
with value $\left( \upsilon _{1},\upsilon _{2}\right) $ specified by
\begin{equation*}
\mathbf{E}\left[ e^{-\lambda _{1}\upsilon _{1}}e^{-\lambda _{2}\upsilon _{2}}%
\right] =1-\psi \left( \lambda _{1},\lambda _{2}\right) .
\end{equation*}
\end{enumerate}
\end{theorem}

We end this section
by the description of the distribution of the MRCA. More precisely,
noticing that
\begin{equation*}
\{Z_{1}(m,n)+Z_{2}(m,n)=1\}=\{\beta _{n}\geq m\},
\end{equation*}%
where the definition of $\beta _{n}$ has been given in 
\eqref{defMRCA}, we deduce from Theorems \ref{T_ZubQ1micro}-\ref%
{T_VatSagmicro} the following statements concerning the birth time $\beta _{n}$ and the type $\mathcal{T}_{n}$ of the
MRCA:

\begin{theorem}
\label{T_MRCA}Let Conditions (\ref{F1})-(\ref{defA21}) be satisfied.

1) If Condition (\ref{Q1NeglQ2}) is fulfilled then for any $%
a\in (0,1)$
\begin{equation*}
\lim_{n\rightarrow \infty }\mathbf{P}\left( \beta _{n}\leq an,%
\mathcal{T}_{n}=2|\mathbf{Z}(n)\neq \mathbf{0}\right) =a;
\end{equation*}

2) If Condition (\ref{VaSag0}) is fulfilled, then for any $t\in
(0,\infty )$ and $a\in (0,1)$,
\begin{equation*}
\lim_{n\rightarrow \infty }\mathbf{P}_{2}\left( \beta _{n}\leq tg^{\ast }(n),%
\mathcal{T}_{n}=2\Big|\mathbf{Z}(n)\neq \mathbf{0}\right) =\frac{\alpha _{2}%
}{1+\alpha _{2}}(1-e^{-(1+\alpha _{2})t/\alpha _{2}}),
\end{equation*}%
and
\begin{equation*}
\lim_{n\rightarrow \infty }\mathbf{P}_{2}(g^{\ast }(n)\ll \beta _{n}\leq an,%
\mathcal{T}_{n}=1|\mathbf{Z}(n)\neq \mathbf{0})=\frac{a}{1+\alpha _{2}}.
\end{equation*}%

3) If Condition (\ref{VaSag}) is fulfilled, then for any $a\in
(0,1)$
\begin{equation*}
\lim_{n\rightarrow \infty }\mathbf{P}_{2}\left( \beta _{n}\leq an|\mathbf{Z}%
(n)\neq \mathbf{0}\right) =1-\frac{1}{1+\alpha _{2}}(1-a)-\frac{\alpha _{2}}{%
1+\alpha _{2}}(1-a)^{((1+\alpha _{2})b^{\alpha _{2}}-1)/\alpha _{2}}.
\end{equation*}%
Moreover, the type of the MRCA satisfies:
\begin{equation*}
\lim_{n\rightarrow \infty }\mathbf{P}_{2}\left( \mathcal{T}_{n}=2|\mathbf{Z}%
(n)\neq \mathbf{0}\right) =\frac{\alpha _{2}b^{\alpha _{2}}}{(1+\alpha
_{2})b^{\alpha _{2}}-1}.
\end{equation*}
\end{theorem}

\begin{remark}
\label{lawbetan2} As a by-product of the proof of Theorem \ref{T_zubQ3}
(see Equation \eqref{dmom}) we can obtain a simple
expression for the death moment $\delta _{n}(2)$ of the last ancestor of
type 2,
\begin{equation*}
\delta _{n}(2):=\min \left\{ m\leq n:Z_{2}\left( m,n\right) =0\right\} .
\end{equation*}%
Namely, under the conditions of Theorem \ref{T_zubQ3} for any $t\in
(0,\infty )$
\begin{equation*}
\lim_{n\rightarrow \infty }\mathbf{P}_{2}\left( \delta _{n}(2)\leq tg^{\ast
}(n)|\mathbf{Z}(n)\neq \mathbf{0}\right) =1-(1+t)^{-1/\alpha _{2}}.
\end{equation*}
\end{remark}

Theorem \ref{T_MRCA} states that the law of the MRCA under
Conditions \eqref{VaSag0} and \eqref{VaSag} essentially differs from its law
under Condition \eqref{Q1NeglQ2}, where it is almost surely of type $2$ and
with birth time uniformly distributed on the time interval $[0,n]$.

\section{Auxiliary results}

\label{sectionauxresults}

In this section we list some known results and prove a number of auxiliary
lemmas. The first statement, which will be useful at
several occasions, is a direct consequence of the representation of
regularly varying functions (see for example \cite{Fel}, Ch.VIII, Section~9).
\label{Lemma1}

\begin{lemma}
\label{L_regular}Let $R(n)$ be a function regularly varying at infinity with
index $-\alpha <0$. If $n\geq m\rightarrow \infty $ then $R(m)\gg R(n)$ if
and only if $m\ll n$.
\end{lemma}

We now reformulate Theorem 1 in \cite{Z},
formula (11) in \cite{Sag87} and Theorem 1 in \cite{VS} as a single
statement. It will be an important tool in  describing 
the local behavior for the reduced processes.

\begin{theorem}
\label{Zub1} Let Conditions (\ref{F1})-(\ref{defA21}) be satisfied.

\begin{enumerate}
\item If Condition \eqref{Q1NeglQ2} holds, then
\begin{equation}
Q_{21}(n)\sim Q_{2}(n),  \label{AsQ1Q22}
\end{equation}%
and for $\lambda _{1}>0,\lambda _{2}>0$
\begin{multline}
\lim_{n\rightarrow \infty }\mathbf{E}\left[ \exp \left\{ -\lambda
_{1}Q_{1}(h^{\ast }(n))Z_{1}(n)-\lambda _{2}Q_{2}(n)Z_{2}(n)\right\} |%
\mathbf{Z}(n)\neq \mathbf{0}\right]  \\
=1-\left( 1+(\lambda _{1}^{1-\alpha _{1}}+\lambda _{2})^{-\alpha
_{2}}\right) ^{-1/\alpha _{2}};  \label{BadStatement}
\end{multline}

\item If Condition \eqref{VaSag0} holds, then
\begin{equation}
Q_{21}(n)\sim Q_{1}^{1/(1+\alpha _{2})}(n)L_{3}(n),\quad
Q_{2}(n)=o(Q_{21}(n)),  \label{SmallQ2}
\end{equation}%
where $L_{3}(n)$ is a slowly varying function as $n\rightarrow \infty ,$ and%
\begin{equation}
Q_{21}^{1+\alpha _{2}}(n)L_{2}(Q_{21}(n))\sim A_{21}Q_{1}(n),  \label{Q21Q1}
\end{equation}%
where $L_{2}$ has been introduced in \eqref{F2}. In
addition, for all $\lambda_{1}\geq 0,\lambda _{2}\geq 0$%
\begin{equation*}
\lim_{n\rightarrow \infty }\mathbf{E}\Big[\exp \left\{ -\lambda
_{1}Q_{1}(n)Z_{1}(n)\right\} |\mathbf{Z}(n)\neq \mathbf{0}\Big]=1-(1+\lambda
_{1}^{-\alpha _{1}})^{-1/(\alpha _{1}(1+\alpha _{2}))},
\end{equation*}%
and
\begin{equation}
\lim_{n\rightarrow \infty }\frac{1}{Q_{2}(n)}\mathbf{E}_{2}\left[ 1-\exp
\left\{ -\lambda _{1}\frac{Q_{2}(n)}{n}Z_{1}(n)-\lambda
_{2}Q_{2}(n)Z_{2}(n)\right\} \right] =\phi (\lambda _{1},\lambda _{2}).
\label{T_sagcom}
\end{equation}%

\item If Condition \eqref{VaSag} holds, then
\begin{equation}
Q_{21}(n)\sim bQ_{2}(n),  \label{surviv}
\end{equation}%
where $b\in \left( 1,\infty \right) $ has been defined in \eqref{defb}. In
addition, for all $\lambda _{1}\geq 0,\lambda _{2}\geq 0$%
\begin{equation}
\lim_{n\rightarrow \infty }\frac{1}{Q_{21}(n)}\mathbf{E}_{2}\left[ 1-\exp
\left\{ -\lambda _{1}\frac{Q_{21}(n)}{n}Z_{1}(n)-\lambda
_{2}Q_{21}(n)Z_{2}(n)\right\} \right] =\psi (\lambda _{1},\lambda _{2}).
\label{equivvatsag}
\end{equation}
\end{enumerate}
\end{theorem}

To simplify notations we agree to write $\mathbf{s}^{\mathbf{k}}=s_{1}^{k_{1}}s_{2}^{k_{2}}$ for any vector $\mathbf{s}=(s_{1},s_{2})$ and
integer-valued vector $\mathbf{k}=(k_{1},k_{2}).$ Besides, for vectors $%
\mathbf{x}=(x_{1},x_{2})$ and $\mathbf{y}=(y_{1},y_{2})$ we set $\mathbf{x}%
\otimes \mathbf{y}=(x_{1}y_{1},x_{2}y_{2})$ and let $\mathbf{1}=(1,1).$
Let us now present a way of expressing the probability generating function
of the reduced process which will be repeatedly used in the proofs and will
allow us to apply Theorem \ref{T_MRCA}. To do this, we need to
introduce some notation. Put, for $i\in \{1,2\}$
\begin{equation*}
\begin{array}{ll}
F_{i}(n;s_{i})=\mathbf{E}\left[ s_{i}^{Z_{i}(n)}|\mathbf{Z}(0)=\mathbf{e}_{i}%
\right] , & F_{21}(n;\mathbf{s})=\mathbf{E}\left[
\mathbf{s}^{\mathbf{Z}(n)}|\mathbf{Z}(0)=\mathbf{e}_{2}\right] , \\
Q_{i}(n;s_{i})=1-F_{i}(n;s_{i}), &
Q_{21}(n;\mathbf{s})=1-F_{21}(n;\mathbf{s}).%
\end{array}%
\end{equation*}%
 By
conditioning on $\mathbf{Z}(m)$, we can write for $m\leq n$
\begin{equation*}
\mathbf{E}_{2}\Big[\mathbf{s}^{\mathbf{Z}(m,n)}\Big]%
=F_{21}(m;1-(1-s_{1})Q_{1}(n-m),1-(1-s_{2})Q_{21}(n-m)).
\end{equation*}%
Hence, setting $\mathbf{Q}_{21}(k)=\left( Q_{1}(k),Q_{21}(k)\right)
,k=0,1,...,$ we get
\begin{equation}
1-\mathbf{E}_{2}\Big[\mathbf{s}^{\mathbf{Z}(m,n)}|\mathbf{Z}(n)\neq
\mathbf{0}\Big]=(Q_{21}(m;\mathbf{1}-(\mathbf{1}-\mathbf{s})\otimes
\mathbf{Q}_{21}(n-m)))/Q_{21}(n).  \label{expressreduced}
\end{equation}%
Thus, to prove Theorems \ref{T_ZubQ1}-\ref{T_VatSag} it is necessary to find
the limit of the right-hand side of~(\ref{expressreduced}) under an
appropriate choice of $m,n$ and $\mathbf{s}$. If $n-m\rightarrow \infty $ it is equivalent to finding
\begin{equation}
\lim_{m,n\rightarrow \infty }Q_{21}(m;\mathbf{e}^{-(\mathbf{1}-\mathbf{s}%
)\otimes \mathbf{Q}_{21}(n-m)})/Q_{21}(n).  \label{ExpressReduced2}
\end{equation}
The end of this section deals with several expressions similar
to \eqref{expressreduced} or (\ref{ExpressReduced2}). These results will be
needed at several occasions in the remaining sections.

\begin{lemma}
\label{L_trick}If $m\ll n$, then, for $\lambda >0$ and $i\in \{1,2\}$
\begin{equation}
Q_{i}( m;e^{-\lambda Q_{i}(n)}) \sim \lambda Q_{i}(n)\sim
Q_{i}( m;1-\lambda Q_{i}(n)) .  \label{asymp1}
\end{equation}
\end{lemma}

\begin{proof}
Let $h_{\lambda }(n)$ be the integer-valued function satisfying%
\begin{equation*}
Q_{i}(h_{\lambda }(n))\leq 1-e^{-\lambda Q_{i}(n)}\leq Q_{i}(h_{\lambda
}(n)-1).
\end{equation*}%
We know by (\ref{AsQ}) that $Q_{i}(h_{\lambda }(n))\sim Q_{i}(h_{\lambda
}(n)-1)$ and%
\begin{equation}
1-e^{-\lambda Q_{i}(n)}\sim \lambda Q_{i}(n)\sim Q_{i}( n\lambda
^{-1/\alpha _{i}} ) , \label{MM}
\end{equation}%
implying $h_{\lambda }(n)\sim n\lambda ^{-1/\alpha _{i}}$ as $n\rightarrow
\infty $. Using the branching property we have%
\begin{eqnarray*}
Q_{i}(m+h_{\lambda }(n)) =Q_{i}(m,F_{i}(h_{\lambda }(n),0))&\leq & Q_{i}(
m;e^{-\lambda Q_{i}(n)}) \\
&\leq &Q_{i}(m,F_{i}(h_{\lambda }(n)-1,0))=Q_{i}(m+h_{\lambda }(n)-1).
\end{eqnarray*}%
Since $m\ll h_{\lambda }(n),$ we get, again by (\ref{AsQ}) that, as $%
n\rightarrow \infty $%
\begin{equation*}
Q_{i}\left( m;e^{-\lambda Q_{i}(n)}\right) \sim Q_{i}(m+h_{\lambda }(n))\sim
Q_{i}(h_{\lambda }(n))\sim Q_{i}\Big( n\lambda ^{-1/\alpha _{i}} \Big)\sim
\lambda Q_{i}(n)
\end{equation*}%
proving the first equivalence in (\ref{asymp1}). The second equivalence
follows from (\ref{MM}).
\end{proof}

\hspace{1cm} \newline

\noindent Lemma \ref{L_zub1} deals with asymptotic results on
the time scale $h^{\ast}(n)$ specified in \eqref{Mu1}.

\begin{lemma}
\label{L_zub1}Let Conditions (\ref{F1})-(\ref{defA21}) and \eqref{Q1NeglQ2}
be satisfied and $\lambda _{1},\lambda _{2},a$ be positive.
Define
\begin{equation}
s=1-\lambda _{1}Q_{1}(h^{\ast }(n)).  \label{DeffS}
\end{equation}%
Then we have the following two  equivalences for large $%
n$,
\begin{equation}
\sum_{k=0}^{n-1}Q_{1}(k;s)\sim \frac{\lambda _{1}^{1-\alpha _{1}}}{A_{21}}%
Q_{2}(n),\quad \sum_{k=0}^{ah^{\ast }(n)-1}Q_{1}(k;s)\sim \frac{\lambda
_{1}^{1-\alpha _{1}}}{A_{21}}\Big(1-(1+a\lambda _{1}^{\alpha
_{1}})^{1-1/\alpha _{1}}\Big)Q_{2}(n),  \label{equiv2}
\end{equation}%
where we agree to understand $ah^{\ast }(n)$ as $\left[ ah^{\ast }(n)\right]
$. Moreover, 
\begin{equation}
\lim_{n\rightarrow \infty }\frac{Q_{21}(ah^{\ast }(n);1-\lambda
_{1}Q_{1}(h^{\ast }(n)),1-\lambda _{2}Q_{2}(n))}{Q_{2}(n)}=\lambda
_{1}^{1-\alpha _{1}}\left( 1-(1+a\lambda _{1}^{\alpha _{1}})^{1-1/\alpha
_{1}}\right) +\lambda _{2}.  \label{C_zub2}
\end{equation}
\end{lemma}

\begin{proof}
We do not prove the first assertion in \eqref{equiv2},
as it may be checked in a similar way as the second one. As in the proof of
the previous lemma, if $h_{\lambda }(n)$ is an integer-valued function
satisfying
\begin{equation*}
Q_{1}(h_{\lambda _{1}}(n))\leq 1-s\leq Q_{1}(h_{\lambda _{1}}(n)-1),
\end{equation*}%
then
\begin{equation}
Q_{1}(k+h_{\lambda _{1}}(n))\leq Q_{1}(k;s)\leq Q_{1}(k+h_{\lambda
_{1}}(n)-1)  \label{compaQ1}
\end{equation}%
for every integer $k$. Equations (\ref{DeffS}) and (\ref{AsQ}) entail that,
as $n\rightarrow \infty $%
\begin{equation}
h_{\lambda _{1}}(n)\sim \lambda _{1}^{-\alpha _{1}}h^{\ast }(n).
\label{eqhlambda}
\end{equation}%
By \eqref{compaQ1} we deduce the inequality
\begin{equation*}
0\leq \sum_{k=0}^{ah^{\ast }(n)-1}Q_{1}(k;s_{1})-\sum_{k=h_{\lambda
_{1}}(n)}^{ah^{\ast }(n)+h_{\lambda _{1}}(n)-1}Q_{1}(k)\leq Q_{1}(h_{\lambda _{1}}(n)-1).
\end{equation*}%
Note that in view of (\ref{AsQ}), as $z\rightarrow \infty $%
\begin{equation*}
\sum_{k=z}^{\infty }Q_{1}(k)\sim \frac{\alpha _{1}}{1-\alpha _{1}}%
zQ_{1}(z)\sim \frac{1}{1-\alpha _{1}}z^{1-1/\alpha _{1}}l_{1}(z).
\end{equation*}%
Hence we conclude that%
\begin{multline*}
\sum_{k=h_{\lambda _{1}}(n)}^{ah^{\ast }(n)+h_{\lambda
_{1}}(n)-1}Q_{1}(k)=\sum_{k=h_{\lambda _{1}}(n)}^{\infty
}Q_{1}(k)-\sum_{k=ah^{\ast }(n)+h_{\lambda _{1}}(n)}^{\infty }Q_{1}(k) \\
\sim \frac{\alpha _{1}}{1-\alpha _{1}}\Big(h_{\lambda
_{1}}(n)Q_{1}(h_{\lambda _{1}}(n))-(ah^{\ast }(n)+h_{\lambda
_{1}}(n))Q_{1}(ah^{\ast }(n)+h_{\lambda _{1}}(n))\Big) \\
\sim \frac{\alpha _{1}}{1-\alpha _{1}}\lambda _{1}^{1-\alpha _{1}}\Big(%
1-(1+a\lambda _{1}^{\alpha _{1}})^{1-1/\alpha _{1}}\Big)h^{\ast
}(n)Q_{1}(h^{\ast }(n)) \\
\sim \frac{1}{A_{21}}\lambda _{1}^{1-\alpha _{1}}\Big(1-(1+a\lambda
_{1}^{\alpha _{1}})^{1-1/\alpha _{1}}\Big)Q_{2}(n),
\end{multline*}%
where we have applied \eqref{eqhlambda}, (\ref{AsQ}) and \eqref{Mu1}.

We now prove \eqref{C_zub2}. By definition, for $\mathbf{s}=(s_{1},s_{2})\in \lbrack 0,1]^{2}$ and $k\in
\mathbb{N}_{0}$,
\begin{equation*}
Q_{21}(k+1;\mathbf{s})=Q_{21}(k;\mathbf{s}%
)-Q_{21}^{1+\alpha _{2}}(k;\mathbf{s})L_{2}\left( Q_{21}(k;\mathbf{s}\right) +\left( A_{21}-\rho \left( k;
\mathbf{s}\right) \right) Q_{1}(k;s_{1}),
\end{equation*}%
where%
\begin{equation*}
\rho \left( k;\mathbf{s}\right) =\rho \left( F_{1}(k;s_{1}),F_{21}(k;\mathbf{s})\right)
\rightarrow 0
\end{equation*}%
as $k\rightarrow \infty $ uniformly in $\mathbf{s}\in \left[ 0,1%
\right]^2 $. Thus,%
\begin{equation*}
Q_{21}(ah^{\ast }(n);\mathbf{s})=Q_{21}(0;\mathbf{s})-\sum_{k=0}^{ah^{\ast }(n)-1}\left[ Q_{21}^{1+\alpha _{2}}(k;
\mathbf{s})L_{2}\left( Q_{21}(k;\mathbf{s})\right) -\left(
A_{21}-\rho \left( k;\mathbf{s}\right) \right) Q_{1}(k;s_{1})%
\right] .
\end{equation*}%
Let, for sufficiently large $n$%
\begin{equation*}
s_{1}=1-\lambda _{1}Q_{1}(h^{\ast }(n))\quad \text{and}\quad s_{2}=1-\lambda
_{2}Q_{2}(n).
\end{equation*}%
Using the inequality $1-c_{1}c_{2}\leq (1-c_{1})+(1-c_{2})$ for $%
c_{1},c_{2}\in \lbrack 0,1]$, we get
\begin{equation*}
Q_{21}(k;\mathbf{s})\leq \mathbf{E}_{2}\left[ Z_{1}(k)\right]
\mathbf{(1-s}_{1}\mathbf{)+1-s}_{2}\mathbf{=}A_{21}k\left( 1-s_{1}\right)
+1-s_{2}.
\end{equation*}%
Adding \eqref{AsQ1Q22} and \eqref{hstar}, we obtain for $k\leq
ah^{\ast }(n)$ and some $C_{1}$ independent of $n$,
\begin{equation*}
Q_{21}(k;\mathbf{s})\leq \lambda _{1}A_{21}kQ_{1}(h^{\ast
}(n))+\lambda _{2}Q_{2}(n)\leq C_{1}Q_{2}(n).
\end{equation*}%
This, in view of the monotonicity of $y^{1+\alpha _{2}}L_{2}\left( y\right) $
as $y\downarrow 0$ and $h^{\ast }(n)\ll n$ yields%
\begin{equation*}
\sum_{k=0}^{ah^{\ast }(n)-1}Q_{21}^{1+\alpha _{2}}(k;\mathbf{s}
)L_{2}\left( Q_{21}(k;\mathbf{s})\right) \leq C_{2}h^{\ast
}(n)Q_{2}^{1+\alpha _{2}}(n)L_{2}(Q_{2}(n))\sim C_{2}\frac{%
h^{\ast }(n)Q_{2}(n)}{\alpha _{2}n}=o(Q_{2}(n)),
\end{equation*}%
where $C_{2}$ is finite and independent of $n$ and we have applied \eqref{ReprQ} to get the equivalence. Finally, recalling that
\begin{equation*}
\sum_{k=0}^{ah^{\ast }(n)-1}\left( A_{21}-\rho \left( k;\mathbf{s%
}\right) \right) Q_{1}(k;s_{1})\sim \lambda _{1}^{1-\alpha _{1}}\Big(%
1-(1+a\lambda _{1}^{\alpha _{1}})^{1-1/\alpha _{1}}\Big)Q_{2}(n)
\end{equation*}%
by \eqref{equiv2} and that
\begin{equation*}
Q_{21}(0;\mathbf{s})=1-s_{2}=\lambda _{2}Q_{2}(n)
\end{equation*}%
we get \eqref{C_zub2} and end the proof of the lemma.
\end{proof}

\hspace{1cm}\newline

We now focus on time scales $h(n)$ larger than $h^\ast(n)$,
still under Condition \eqref{Q1NeglQ2}.

\begin{lemma}
\label{L_range}Let Conditions (\ref{F1})-(\ref{defA21}) and (\ref{Q1NeglQ2})
be satisfied and $h(n)$ be an integer-valued function such that $h^{\ast
}(n)\ll h(n)$. Then, for any $s\in \lbrack 0,1)$
\begin{equation*}
\lim_{n\rightarrow \infty }\sup_{m\leq n-h(n)}\mathbf{E}_{2}\left[
1-s^{Z_{1}(m,n)}|\mathbf{Z}(n)\neq \mathbf{0}\right] =\lim_{n\rightarrow
\infty }\sup_{m\leq n-h(n)}\frac{\mathbf{E}_{2}\left[ 1-s^{Z_{1}(m,n)}\right]
}{Q_{21}(n)}=0.
\end{equation*}
\end{lemma}

\begin{proof}
Since $Z_{1}(m,n)$ is monotone increasing in $m$ given $n,$ it is sufficient
to consider $m=n-h(n)$. By conditioning on $Z_{1}(m)$ we get
\begin{equation*}
\mathbf{E}_{2}\left[ s^{Z_{1}(m,n)}\right] =\mathbf{E}_{2}\left[ \left(
F_{1}(n-m;0)+s\left( 1-F_{1}(n-m;0)\right) \right) ^{Z_{1}(m)}\right] \geq
\mathbf{E}_{2}\left[ F_{1}^{Z_{1}(m)}(n-m;0)\right] .
\end{equation*}%
Let $\zeta (k),k=1,2,...$ be the total amount of type $1$ particles at
moment $k$ produced by all particles of type $2$ existing in the branching
process at moment $k-1$. Then,%
\begin{equation*}
\mathbf{E}_{2}\left[ F_{1}^{Z_{1}(m)}(n-m;0)\right] =\mathbf{E}_{2}\left[
\prod_{k=0}^{m-1}F_{1}^{\zeta (k)}(m-k;F_{1}(n-m;0))\right] =\mathbf{E}_{2}%
\left[ \prod_{k=0}^{m-1}F_{1}^{\zeta (k)}(n-k;0)\right] .
\end{equation*}%
By iteration, we also deduce for every $k=1,2,...$:
\begin{equation*}
\mathbf{E}_{2}\left[ \zeta (k)\right] =\mathbf{E}_{2}\left[ \mathbf{E}_{2}%
\left[ \zeta (k)|Z_{2}(k-1)\right] \right] =A_{21}\mathbf{E}_{2}\left[
Z_{2}(k-1)\right] =A_{21}.
\end{equation*}%
Using the estimates above we obtain%
\begin{multline*}
\mathbf{E}_{2}\left[ 1-s^{Z_{1}(m,n)}\right] \leq \mathbf{E}_{2}\left[
1-\prod_{k=0}^{m-1}F_{1}^{\zeta (k)}(n-k;0)\right] \leq \mathbf{E}_{2}\left[
\sum_{k=0}^{m-1}\zeta (k)Q_{1}(n-k)\right]  \\
\leq A_{21}\sum_{k=n-m}^{\infty }Q_{1}(k)=O\left( h(n)Q_{1}(h(n))\right) ,
\end{multline*}%
where at the last stage we have used (\ref{AsQ}) and the fact that, as $%
n\rightarrow \infty $%
\begin{equation*}
Q_{1}(n)=o\left( n^{-1}Q_{2}(n)\right) =o\left( n^{-1-\varepsilon }\right)
\end{equation*}%
for some $\varepsilon >0$. Since $h^{\ast }(n)\ll h(n)$ as $n\rightarrow
\infty ,$ we have by Lemma \ref{L_regular}, (\ref{Mu1}) and (\ref{AsQ1Q22})
that
\begin{equation*}
h(n)Q_{1}(h(n))=o\left( h^{\ast }(n)Q_{1}(h^{\ast }(n))\right) =o\left(
Q_{21}(n)\right) ,
\end{equation*}%
proving the lemma.
\end{proof}

\hspace{1cm}\newline

We now state asymptotic results on the time scale $g^{\ast }(n)$ specified
by \eqref{defgast}.

\begin{lemma}
\label{L_propg}Let Conditions (\ref{F1})-(\ref{defA21}) and (\ref{VaSag0})
be satisfied and $(s_1,s_2)$ be in $[0,1]^2$. Then

\begin{enumerate}
\item
\begin{equation}
Q_{2}(g^{\ast }(n))\sim \alpha _{2}A_{21}Q_{1}(n)g^{\ast }(n);
\label{eqgast}
\end{equation}

\item If $m\ll g^{\ast }(n),$ then%
\begin{equation}
\mathbf{E}_{2}\left[ 1-s_{1}^{Z_{1}(m,n)}\right] =o(Q_{21}(n));
\label{Step1}
\end{equation}

\item If $g^{\ast }(n)\ll m\leq n,$ then
\begin{equation}
\mathbf{E}_{2}\left[ 1-s_{2}^{Z_{2}(m,n)}\right] =o(Q_{21}(n)).
\label{Step2}
\end{equation}
\end{enumerate}
\end{lemma}

\begin{proof}
First observe that according to (\ref{defgast2}), $g^{\ast
}(n)\ll n$, and that (\ref{eqgast}) is a direct consequence of~(\ref{ReprQ})
and (\ref{Q21Q1}). To prove the remaining statements note that by
the branching property
\begin{equation}
\mathbf{E}_{2}\left[ 1-s_{1}^{Z_{1}(m,n)}\right] \leq \mathbf{E}_{2}\left[
Z_{1}(m,n)\right] =\mathbf{E}_{2}\left[ Z_{1}(m)\right] Q_{1}\left(
n-m\right) =A_{21}mQ_{1}\left( n-m\right) .  \label{Neg1}
\end{equation}%
Recalling the condition $m\ll g^{\ast }(n),$ Lemma \ref{L_regular} and (\ref%
{eqgast}), we conclude that%
\begin{equation*}
mQ_{1}\left( n-m\right) \ll g^{\ast }(n)Q_{1}\left( n-g^{\ast }(n)\right)
\sim g^{\ast }(n)Q_{1}\left( n\right) \sim \frac{Q_{2}(g^{\ast }(n))}{\alpha
_{2}A_{21}}\sim \frac{Q_{21}(n)}{\alpha _{2}A_{21}}.
\end{equation*}%
This proves (\ref{Step1}). To justify (\ref{Step2}) it is sufficient to
observe that, for $g^{\ast }(n)\ll m$,
\begin{equation*}
\mathbf{E}_{2}\left[ 1-s_{2}^{Z_{2}(m,n)}\right] \leq Q_{2}(m)\ll
Q_{2}(g^{\ast }(n))\sim Q_{21}(n).
\end{equation*}%
This ends the proof of the lemma.
\end{proof}

\hspace{1cm}\newline

To end this section we mention a classical result on
branching processes. Recall Definition \ref{defprocessX} and define $\{%
\mathbf{M}=(M_{1}(t),M_{2}(t)),t\geq 0\}=\{\mathbf{\nu }%
,(g_{1}(s_{1}),\mu _{1}),(g_{21}(\mathbf{s}),\mu _{2})\}$. Let
\begin{equation*}
f_{1}(t;s_{1})=\mathbf{E}\Big[s_{1}^{M_{1}(t)}|\mathbf{M}(0)=\mathbf{e}_{1}%
\Big],\quad f_{21}(t;\mathbf{s})=\mathbf{E}\Big[%
\mathbf{s}^{\mathbf{M}(t)}|\mathbf{M}(0)=\mathbf{e}_{2}\Big].
\end{equation*}%
Then we have the following statement (see, for instance, \cite%
{Sev72}, Ch. IV, Section 3 or \cite{AN72}, Ch. V, Section 7).

\begin{lemma}
\label{lemmaathreyaney} The pair of functions $(f_{1},f_{21})$ 
is the unique solution of
\begin{equation*}
\left\{%
\begin{array}{lll}
\frac{\partial f_{21}\left( t;\mathbf{s}\right) }{\partial t} & = & \mu
_{2}\left( g_{21}\left( f_{1}\left( t;s_{1}\right) ,f_{21}\left(
t;\mathbf{s}\right) \right) -f_{21}\left( t;\mathbf{s}\right) \right) , \\
\frac{\partial f_{1}(t;s_{1})}{\partial t} & = & \mu
_{1}(g_{1}(f_{1}(t;s_{1}))-f_{1}(t;s_{1}))%
\end{array}%
\right.
\end{equation*}
with initial conditions $\ f_{1}\left( 0;s_{1}\right) =s_{1}$ and $%
f_{21}(0;\mathbf{s})=s_{2}.$
\end{lemma}

\section{Proof of Theorem \protect\ref{T_ZubQ1}}

\label{proofTh1}

Throughout this section we  assume Conditions \eqref{F1}%
-\eqref{defA21} and \eqref{Q1NeglQ2}, and that $(s_1,s_2)$ belongs to $(0,1)^2$. 

\noindent \textbf{Point (0):} Let $m\ll n$. In view of (\ref{Q1NeglQ2}) and Lemma \ref{L_range}
\begin{equation*}
0\leq \mathbf{E}\left[ s_{2}^{Z_{2}(m,n)}|\mathbf{Z}(n)\neq \mathbf{0}\right]
-\mathbf{E}\left[ \mathbf{s}^{\mathbf{Z}(m,n)}|\mathbf{Z}%
(n)\neq \mathbf{0}\right] \leq \mathbf{E}\left[ 1-s_{1}^{Z_{1}(m,n)}|\mathbf{%
Z}(n)\neq \mathbf{0}\right] =o(1).
\end{equation*}%
Thus,%
\begin{equation*}
\mathbf{E}\left[\mathbf{s}^{\mathbf{Z}(m,n)}|\mathbf{Z}(n)\neq
\mathbf{0}\right] =1-\mathbf{E}\left[ 1-s_{2}^{Z_{2}(m,n)}\right]
/Q_{21}(n)+o(1).
\end{equation*}%
Further, using the branching property and Lemma \ref{L_trick}
we get%
\begin{eqnarray*}
\mathbf{E}\left[ 1-s_{2}^{Z_{2}(m,n)}\right] = Q_{2}\left( m;1-\left(
1-s_{2}\right) Q_{2}(n-m)\right) \sim \left( 1-s_{2}\right) Q_{2}(n-m)\sim
\left( 1-s_{2}\right) Q_{2}(n).
\end{eqnarray*}%
This, on account of (\ref{AsQ1Q22}) implies (\ref{Zub11}).

\noindent \textbf{Point (1):} Let $a$ be in $(0,1)$. As in the previous
case, by (\ref{Q1NeglQ2}) and Lemma \ref{L_range}, as $n\rightarrow \infty $,
\begin{equation*}
\mathbf{E}\left[ \mathbf{s}^{\mathbf{Z}(an,n)}|\mathbf{Z}%
(n)\neq \mathbf{0}\right] =\mathbf{E}\left[ s_{2}^{Z_{2}(an,n)}|\mathbf{Z}%
(n)\neq \mathbf{0}\right] +o(1).
\end{equation*}%
Moreover, in view of (\ref{monotype}) and (\ref%
{AsQ1Q22}), as $n\rightarrow \infty $%
\begin{eqnarray*}
1-\mathbf{E}\left[ s_{2}^{Z_{2}(an,n)}|\mathbf{Z}(n)\neq \mathbf{0}\right]
&=&\frac{Q_{2}\left( an;1-(1-s_{2})Q_{21}((1-a)n))\right) }{Q_{21}(n)} \\
&\sim &\frac{Q_{2}\left( an,e^{-(1-s_{2})(a/(1-a))^{1/\alpha
_{2}}Q_{2}(an))}\right) }{a^{1/\alpha _{2}}Q_{2}(an)} \\
&\sim &a^{-1/\alpha _{2}}\Big( 1+\Big( (1-s_{2})\Big( \frac{a}{1-a}\Big) %
^{1/\alpha _{2}}\Big) ^{-\alpha _{2}}\Big) ^{-1/\alpha _{2}} \\
&=&\left( a+(1-a)(1-s_{2})^{-\alpha _{2}}\right) ^{-1/\alpha _{2}}.
\end{eqnarray*}%
This completes the proof of point (1).

\noindent \textbf{Point (2):} Let $n-m=h(n),$ where $h(n)$ is an
integer-valued function such that $h^{\ast }(n)\ll h(n)\ll n.$ Similarly to
the previous point we have by (\ref{Q1NeglQ2}) and Lemma~\ref{L_range}%
\begin{equation*}
0\leq \mathbf{E}\left[ s_{2}^{Z_{2}(m,n)}|\mathbf{Z}(n)\neq \mathbf{0}\right]
-\mathbf{E}\left[ \mathbf{s}^{\mathbf{Z}(m,n)}|\mathbf{Z}%
(n)\neq \mathbf{0}\right] =o(1)
\end{equation*}%
as $n\rightarrow \infty $. Now following the line of proving point (1) and
letting
\begin{equation*}
1-s_{2}=1-\exp \left\{ -\lambda _{2}\frac{Q_{2}(n)}{Q_{21}(h(n))}\right\}
\sim \lambda _{2}\frac{Q_{2}(n)}{Q_{21}(h(n))}
\end{equation*}%
we get, using again (\ref{monotype}),
\begin{eqnarray*}
\lim_{n\rightarrow \infty }\mathbf{E}\Big[1-\mathbf{s}^{\mathbf{Z}(n-h(n),n)}%
\Big|\mathbf{Z}(n)\neq \mathbf{0}\Big]& =&\lim_{n\rightarrow \infty }\frac{%
Q_{2}(n-h(n);1-(1-s_{2})Q_{21}(h(n)))}{Q_{21}(n)} \\
&=&\lim_{n\rightarrow \infty }\frac{Q_{2}(n-h(n);1-\lambda _{2}Q_{2}(n))}{%
Q_{2}(n)}=\left( 1+\lambda _{2}^{-\alpha _{2}}\right) ^{-1/\alpha _{2}}.
\end{eqnarray*}

\noindent \textbf{Point (3):} Let $t$ be positive and $%
m=n-th^{\ast }(n)$. We need to find the limit of the right hand side of (\ref%
{expressreduced}) with
\begin{equation*}
1-s_{2}=\exp \left\{ -\lambda _{2}\frac{Q_{2}(n)}{Q_{21}(th^{\ast }(n))}%
\right\} \sim \lambda _{2}\frac{Q_{2}(n)}{Q_{21}(th^{\ast }(n))}.
\end{equation*}%
From  (\ref{BadStatement}) we know that
\begin{equation*}
\lim_{m\rightarrow \infty }\frac{Q_{21}(m;e^{-\lambda _{1}Q_{1}(h^{\ast
}(m))};e^{-\lambda _{2}Q_{2}(m)})}{Q_{21}(m)}=\left( 1+\left( \lambda
_{1}^{1-\alpha _{1}}+\lambda _{2}\right) ^{-\alpha _{2}}\right) ^{-1/\alpha
_{2}}.
\end{equation*}%
Hence, by taking a fixed $s_{1}\in \lbrack 0,1)$ and adding %
\eqref{AsQ} we deduce for $m=n-th^{\ast }(n)$ that
\begin{multline*}
\lim_{n\rightarrow \infty }\frac{Q_{21}(n-th^{\ast }(n);\mathbf{1}-(\mathbf{1%
}-\mathbf{s})\otimes \mathbf{Q}_{21}(th^{\ast }(n)))}{Q_{21}(n)} \\
=\lim_{n\rightarrow \infty }\frac{Q_{21}(n-th^{\ast }(n);1-t^{-1/\alpha
_{1}}(1-s_{1})Q_{1}(h^{\ast }(n)),1-\lambda _{2}Q_{2}(n))}{Q_{21}(n)} \\
=(1+(t^{1-1/\alpha _{1}}(1-s_{1})^{1-\alpha _{1}}+\lambda _{2})^{-\alpha
_{2}})^{-1/\alpha _{2}},
\end{multline*}
which is the statement of point (3).

\noindent \textbf{Point (4):} Take $n-m=h(n)$ with $1\ll h(n)\ll h^{\ast }(n)
$,
\begin{equation*}
s_{1}=\exp \left\{ -\lambda _{1}\frac{Q_{1}(h^{\ast }(n))}{Q_{1}(h(n))}%
\right\} \quad \text{and}\quad s_{2}=\exp \left\{ -\lambda _{2}\frac{Q_{2}(n)%
}{Q_{21}(h(n))}\right\} .
\end{equation*}%
Using (\ref{BadStatement}) we get
\begin{eqnarray*}
\lim_{n\rightarrow \infty }\frac{Q_{21}(n-h(n);\mathbf{1}-(\mathbf{1}-%
\mathbf{s})\otimes \mathbf{Q}_{21}(h(n)))}{Q_{21}(n)} &=&\lim_{n\rightarrow
\infty }\frac{Q_{21}(n-h(n);e^{-\lambda _{1}Q_{1}(h^{\ast }(n))},e^{-\lambda
_{2}Q_{2}(n)})}{Q_{21}(n-h(n))} \\
&=&\left( 1+(\lambda _{1}^{1-\alpha _{1}}+\lambda _{2})^{-\alpha
_{2}}\right) ^{-1/\alpha _{2}}.
\end{eqnarray*}

This ends the proof of Theorem \ref{T_ZubQ1}.

\section{Proof of Theorem \protect\ref{T_zubQ3}}

\label{proofTh3+Cor4}

Throughout this section we assume that Conditions \eqref{F1}-%
\eqref{defA21} and \eqref{VaSag0} are in force and that $(s_{1},s_{2})$ belongs to $(0,1)^{2}$.

\noindent \textbf{Point (0):} Let $m\ll g^{\ast }(n)$ with $%
g^{\ast }(n)$ specified by~\eqref{defgast}. Applying (\ref{Step1}) we get, as $n\rightarrow \infty $%
\begin{equation*}
0\leq \mathbf{E}\left[ 1-\mathbf{s}^{\mathbf{Z}(m,n)}\right] -%
\mathbf{E}\left[ 1-s_{2}^{Z_{2}(m,n)}\right] \leq \mathbf{E}\left[
1-s_{1}^{Z_{1}(m,n)}\right] =o\left( Q_{21}(n)\right) .
\end{equation*}%
Further, recalling \eqref{defgast} and Lemma \ref{L_trick} we
see that%
\begin{eqnarray*}
\mathbf{E}\left[ 1-s_{2}^{Z_{2}(m,n)}\right] & = & Q_{2}\left(
m;1-(1-s_{2})Q_{21}(n-m))\right) \\
&\sim & Q_{2}\left( m;1-(1-s_{2})Q_{2}(g^{\ast }(n))\right) \sim
(1-s_{2})Q_{2}(g^{\ast }(n))\sim (1-s_{2})Q_{21}(n).
\end{eqnarray*}
Thus,
\begin{eqnarray*}
\lim_{n\rightarrow \infty }\mathbf{E}[1-s_{2}^{Z_{2}(m,n)}|\mathbf{Z}(n)\neq
\mathbf{0}] =\lim_{n\rightarrow \infty }\frac{\mathbf{E}[1-
\mathbf{s}^{\mathbf{Z}(m,n)}]}{Q_{21}(n)} =\lim_{n\rightarrow \infty }\frac{%
\mathbf{E}[1-s_{2}^{Z_{2}(m,n)}]}{Q_{21}(n)}=1-s_{2},
\end{eqnarray*}
which proves point (0).

\noindent \textbf{Point (1):} In view of (\ref{AsQ}) and 
\eqref{defgast}, for any $t>0$
\begin{equation*}
Q_{21}(n-tg^{\ast }(n))\sim Q_{21}(n)\sim Q_{2}(g^{\ast }(n))\sim
t^{1/\alpha _{2}}Q_{2}(tg^{\ast }(n))
\end{equation*}%
and by \eqref{eqgast} and \eqref{AsQ}
\begin{equation*}
Q_{1}(n-tg^{\ast }(n))\sim Q_{1}(n)\sim \frac{Q_{2}(g^{\ast }(n))}{\alpha
_{2}A_{21}g^{\ast }(n)}\sim \frac{t^{1/\alpha _{2}+1}}{\alpha _{2}A_{21}}%
\frac{Q_{2}(tg^{\ast }(n))}{tg^{\ast }(n)}.
\end{equation*}

Then, by using \eqref{expressreduced}, \eqref{defgast} and 
\eqref{T_sagcom} with $n$ replaced by $tg^{\ast }(n)$ we obtain
\begin{multline*}
\mathbf{E}_{2}[1-\mathbf{s}^{\mathbf{Z}(tg^{\ast }(n),n)}|\mathbf{Z}(n)\neq
\mathbf{0}] = \frac{Q_{21}(tg^\ast(n),1-(\mathbf{1}-\mathbf{s}%
)\otimes \mathbf{Q}_{21}(n-tg^{\ast }(n)) )}{Q_{21}(n)}  \\
\sim \frac{\mathbf{E}_{2}\left[ 1-\exp \left\{ -(1-s_{1})\frac{t^{1/\alpha
_{2}+1}}{\alpha _{2}A_{21}}\frac{Q_{2}(tg^{\ast }(n))}{tg^{\ast }(n)}%
Z_{1}(tg^{\ast }(n))-(1-s_{2})t^{1/\alpha _{2}}Q_{2}(tg^{\ast
}(n))Z_{2}(tg^{\ast }(n))\right\} \right] }{t^{1/\alpha _{2}}Q_{2}(tg^{\ast
}(n))} \\
\sim t^{-1/\alpha _{2}}\phi \Big((1-s_{1})\frac{t^{1/\alpha _{2}+1}}{\alpha
_{2}A_{21}},(1-s_{2})t^{1/\alpha _{2}}\Big)
\end{multline*}%
justifying the statement of point (1).

\noindent \textbf{Point (3):} If $g^{\ast }(n)\ll m\leq n$ then, according
to (\ref{Step2})
\begin{equation*}
\left\vert \mathbf{E}_{2}\left[ \left( 1^{Z_{2}(m,n)}-0^{Z_{2}(m,n)}\right)
s_{1}^{Z_{1}(m,n)}\right] \right\vert \leq \mathbf{E}_{2}\left[
1-0^{Z_{2}(m,n)}\right] =o\left( Q_{21}(n)\right) .
\end{equation*}%
Hence we obtain the same limit for \eqref{T_zubQ3point4} independently of
the choice of $s_{2}$. We choose $s_{2}$ satisfying
\begin{equation*}
(1-s_{2})Q_{21}((1-a)n)\sim Q_{21}\left( \frac{(1-a)n}{(1-s_{1})^{\alpha
_{1}}}\right) \sim (1-s_{1})^{1/(1+\alpha _{2})}(1-a)^{-1/(\alpha
_{1}(1+\alpha _{2}))}Q_{21}(n),
\end{equation*}%
where we have applied \eqref{SmallQ2}. Then, using the
branching property, \eqref{AsQ} and \eqref{SmallQ2}, we
get
\begin{multline*}
\frac{Q_{21}(an;\mathbf{1}-(\mathbf{1}-\mathbf{s})\otimes \mathbf{Q}%
_{21}((1-a)n))}{Q_{21}(n)} \sim \frac{Q_{21}(an;\mathbf{1}-\mathbf{Q}%
_{21}((1-a)n/(1-s_{1})^{\alpha _{1}}))}{Q_{21}(n)} \\
=\frac{Q_{21}(an+(1-a)n/(1-s_{1})^{\alpha _{1}}))}{Q_{21}(n)}\sim \Big(a+%
\frac{(1-a)}{(1-s_{1})^{\alpha _{1}}}\Big)^{-1/(\alpha _{1}(1+\alpha _{2}))}.
\end{multline*}%
Adding \eqref{expressreduced}, this proves point (3).

\noindent \textbf{Point (2):} We know from point (1) that for every $t>0$,
\begin{equation*}
\lim_{n\rightarrow \infty }\mathbf{P}\left( Z_{2}(tg^{\ast }(n),n)=0|\mathbf{%
Z}(n)\neq \mathbf{0}\right) =1-t^{-1/\alpha _{2}}\phi (0,t^{1/\alpha _{2}}).
\end{equation*}%
Setting $q(x)=\phi (0,x)/x$, $x>0$, we deduce by applying \eqref{defphi}
that $q(x)$ satisfies
\begin{equation*}
q^{\prime }(x)=-x^{\alpha _{2}-1}q^{1+\alpha _{2}}(x),
\end{equation*}%
with initial condition $q(0)=\lim_{x\downarrow 0}q(x)=1$. This equation has
an explicit solution and we obtain
\begin{equation}
\lim_{n\rightarrow \infty }\mathbf{P}_{2}\left( Z_{2}(tg^{\ast }(n),n)=0|%
\mathbf{Z}(n)\neq \mathbf{0}\right) =1-q(t^{1/\alpha
_{2}})=1-(1+t)^{-1/\alpha _{2}}.  \label{dmom}
\end{equation}%
In particular,
\begin{equation}
\lim_{t\rightarrow \infty }\lim_{n\rightarrow \infty }\mathbf{P}_{2}\left(
Z_{2}(tg^{\ast }(n),n)=0|\mathbf{Z}(n)\neq \mathbf{0}\right) =1,
\label{Z2mnnull}
\end{equation}%
and for  $m\gg g^{\ast }(n)$,
\begin{equation}
\lim_{m,n\rightarrow \infty }\mathbf{P}_{2}\left( Z_{2}(m,n)=0|\mathbf{Z}%
(n)\neq \mathbf{0}\right) =1.  \label{Z2mnnull2}
\end{equation}%
We now take $r(x)=\phi (x,0)/x^{1/(1+\alpha _{2})},x>0$. Then from point (1)
we have
\begin{equation*}
\lim_{n\rightarrow \infty }\mathbf{E}_{2}\left[ s_{1}^{Z_{1}((\alpha
_{2}A_{21}x/(1-s_{1}))^{\alpha _{2}/(1+\alpha _{2})}g^{\ast }(n),n)}|\mathbf{%
Z}(n)\neq \mathbf{0}\right] =1-\Big(\frac{1-s_{1}}{\alpha _{2}A_{21}}\Big)%
^{1/(\alpha _{2}+1)}r(x).
\end{equation*}%
As $Z_{1}(m,n)$ is non-decreasing with respect to $m$ and we are 
looking for non-negative quantities, it follows that $r$ is a non-negative,
non-decreasing and bounded function. Hence it admits a non-negative limit at
infinity. Moreover from \eqref{Z2mnnull} we see that
\begin{equation*}
\lim_{t\rightarrow \infty }\lim_{n\rightarrow \infty }\mathbf{E}\left[
0^{Z_{1}((\alpha _{2}A_{21}t)^{\alpha _{2}/(1+\alpha _{2})}g^{\ast }(n),n)}|%
\mathbf{Z}(n)\neq \mathbf{0}\right] =0
\end{equation*}%
in view of $Z_{1}(m,n)+Z_{2}(m,n)\geq 1$. Hence $\lim_{x\rightarrow \infty
}r(x)=\left( \alpha _{2}A_{21}\right) ^{1/(\alpha _{2}+1)}$ and
\begin{equation*}
\lim_{t\rightarrow \infty }\lim_{n\rightarrow \infty }\mathbf{E}%
[s_{1}^{Z_{1}(tg^{\ast }(n),n)}|\mathbf{Z}(n)\neq \mathbf{0}%
]=1-(1-s_{1})^{1/(1+\alpha _{2})}.
\end{equation*}%
From point (3) we also know that
\begin{equation*}
\lim_{a\rightarrow 0}\lim_{n\rightarrow \infty }\mathbf{E}%
[s_{1}^{Z_{1}(an,n)}|\mathbf{Z}(n)\neq \mathbf{0}]=1-(1-s_{1})^{1/(1+\alpha
_{2})}.
\end{equation*}%
As $Z_{1}(m,n)$ is non-decreasing with respect to $m$, we end the proof of
point (2) with \eqref{Z2mnnull2}.

\noindent \textbf{Point (4):} The proof is similar to that of point (3) and
we give less details. In particular, the choice of $s_{2}$ has no influence
on the limit, as there is no type $2$ individuals anymore in the reduced
process. Let $n-m=h(n)$ with $1\ll h(n)\ll n$. Set%
\begin{equation*}
1-s_{1}=1-\exp \left\{ -\lambda _{1}\frac{Q_{1}(n)}{Q_{1}(h(n))}\right\}
\sim \lambda _{1}\frac{Q_{1}(n)}{Q_{1}(h(n))}
\end{equation*}%
and choose $s_{2}$ such that
\begin{equation*}
(1-s_{2})Q_{21}(h(n))\sim Q_{21}(\lambda _{1}^{-\alpha _{1}}n).
\end{equation*}%
Then by \eqref{expressreduced}, the branching property, %
\eqref{AsQ} and \eqref{SmallQ2} we conclude that
\begin{multline*}
\mathbf{E}[1-\mathbf{s}^{\mathbf{Z}(m,n)}|\mathbf{Z}(n)\neq \mathbf{0}]=%
\frac{Q_{21}(n-h(n);\mathbf{1}-(\mathbf{1}-\mathbf{s})\otimes \mathbf{Q}%
_{21}(h(n)))}{Q_{21}(n)} \\
\sim \frac{Q_{21}(n-h(n);\mathbf{1}-\mathbf{Q}_{21}(\lambda _{1}^{-\alpha
_{1}}n))}{Q_{21}(n)} =\frac{Q_{21}((1+\lambda _{1}^{-\alpha _{1}})n-h(n))}{%
Q_{21}(n)}\sim (1+\lambda _{1}^{-\alpha _{1}})^{-\frac{1}{\alpha
_{1}(1+\alpha _{2})}},
\end{multline*}%
which ends the proof of Theorem \ref{T_zubQ3}.

\section{Proof of Theorem \protect\ref{T_VatSag}}

\label{proofTh5}

Throughout this section we assume that Conditions \eqref{F1}-%
\eqref{defA21} and \eqref{VaSag} are satisfied and that $(s_{1},s_{2})$ {%
belongs to $(0,1)^{2}$.

\noindent \textbf{Point (0):} Using (\ref{Neg1}), Condition %
\eqref{VaSag} and \eqref{surviv} we obtain for $m\ll n$
\begin{equation*}
\mathbf{E}\left[ 1-s_{1}^{Z_{1}(m,n)}\right] \leq A_{21}mQ_{1}\left(
n-m\right) \ll nQ_{1}\left( n\right) \sim \sigma Q_{2}(n)\sim \sigma
b^{-1}Q_{21}(n),
\end{equation*}%
implying%
\begin{equation*}
0\leq \mathbf{E}\left[1-\mathbf{s}^{\mathbf{Z}(m,n)}%
\right] -\mathbf{E}\left[ 1-s_{2}^{Z_{2}(m,n)}\right]\leq \mathbf{E}\left[
1-s_{1}^{Z_{1}(m,n)}\right] =o\left( Q_{21}(n)\right) .
\end{equation*}%
Further, by Lemma \ref{L_trick} and (\ref{surviv})
\begin{equation*}
\mathbf{E}\left[ 1-s_{2}^{Z_{2}(m,n)}\right] =Q_{2}\left(
m;1-(1-s_{2})Q_{21}(n-m)\right) \sim (1-s_{2})Q_{21}(n),
\end{equation*}%
giving (\ref{Negl2}).

\noindent \textbf{Point (1):} The result is a direct consequence
of \eqref{equivvatsag}, \eqref{expressreduced}, and of the equivalence $%
Q_1(n) \sim \sigma Q_{21}(n)/b $ under Condition \eqref{VaSag}.

\noindent \textbf{Point (2):} Again it is a direct consequence of
expressions \eqref{equivvatsag} and \eqref{expressreduced}. As it is very
similar to the proofs of points (3) and (4) of Theorem \ref{T_zubQ3}, we do
not give the details.

\section{Convergence of finite dimensional distributions\label{S_finiteDimen}}

In this section we study the limiting behavior of the finite-dimensional
distributions of the properly scaled reduced process $\left\{ \mathbf{Z}%
(m,n),0\leq m\leq n\right\} $. To simplify notation we let

\begin{equation*}
J_{i}^{\,(m,n)}(\mathbf{s})=\mathbf{E}\left[ \mathbf{s}^{\mathbf{Z}(m,n)}|%
\mathbf{Z}(n)\neq \mathbf{0},\mathbf{Z}(0)=\mathbf{e}_{i}\right] , \quad
i=1,2
\end{equation*}%
and, given $0\leq k_{0}<k_{1}<...<k_{p}\leq n$ set $\mathbf{k}%
=(k_{0},k_{1},...,k_{p})$. Putting $\mathbf{S}_{l}=\left(
s_{1l},s_{2l}\right) $ denote%
\begin{eqnarray*}
\hat{J}_{2}^{(k_{0},k_{1},...,k_{p},n)}(\mathbf{S}_{1},...,\mathbf{S}_{p})
&=&\hat{J}_{2}^{(\mathbf{k},n)}(\mathbf{S}_{1},...,\mathbf{S}_{p})=\mathbf{E}%
\left[ \prod_{l=1}^{p}\mathbf{S}_{l}^{\mathbf{Z}(k_{l},n)}\Big|\mathbf{Z}%
(k_{0},n)=\mathbf{e}_{2}\right] , \\
\hat{J}_{1}^{(k_{0},k_{1},...,k_{p},n)}(s_{11},...,s_{1p}) &=&\hat{J}_{1}^{(%
\mathbf{k},n)}(s_{11},...,s_{1p})=\mathbf{E}\left[
\prod_{l=1}^{p}s_{1l}^{Z_{1}(k_{l},n)}\Big|\mathbf{Z}(k_{0},n)=\mathbf{e}_{1}%
\right],
\end{eqnarray*}%
and%
\begin{equation*}
\mathbf{\hat{J}}^{(\mathbf{k},n)}(\mathbf{S}_{1},...,\mathbf{S}_{p})=\left(
\hat{J}_{1}^{(\mathbf{k},n)}(s_{11},...,s_{1p}),\hat{J}_{2}^{(\mathbf{k},n)}(%
\mathbf{S}_{1},...,\mathbf{S}_{p})\right) .
\end{equation*}

The next statement is a simple observation following from Corollary 2 in~%
\cite{VD06}.

\begin{lemma}
\label{L_convol}For any $0\leq k_{0}<k_{1}<...<k_{p}\leq n$ we have%
\begin{multline*}
\hat{J}_{2}^{(\mathbf{k},n)}(\mathbf{S}_{1},...,\mathbf{S}_{p})=\hat{J}%
_{2}^{(k_{0},k_{1},n)}\left( \mathbf{S}_{1}\otimes \mathbf{\hat{J}}%
^{(k_{1},k_{2},...,k_{p},n)}(\mathbf{S}_{2},...,\mathbf{S}_{p})\right)  \\
=J_{2}^{(k_{1}-k_{0},n-k_{0})}\left( \mathbf{S}_{1}\otimes \mathbf{J}%
^{\left( k_{2}-k_{1},n-k_{1}\right) }\left( \mathbf{S}_{2}\otimes \left(
...\left( \mathbf{S}_{p-1}\otimes \mathbf{J}^{\left(
k_{p}-k_{p-1},n-k_{p-1}\right) }(\mathbf{S}_{p})\right) ...\right) \right)
\right) .
\end{multline*}%
In particular, if $\mathbf{k}=(0,k_{1},k_{2}),$ then%
\begin{equation*}
\hat{J}_{2}^{\,(0,k_{1},k_{2},n)}(\mathbf{S}_{1},\mathbf{S}_{2})=J_{2}^{\,(k_{1},n)}(\mathbf{S}_{1}\otimes \mathbf{J}^{\left( 
k_{2}-k_{1},n-k_{1}\right) }(\mathbf{S}_{2})).
\end{equation*}
\end{lemma}

\begin{remark}
It follows from Lemma \ref{L_convol} that the process $\left\{ \mathbf{Z}%
(m,n),0\leq m\leq n|\,\mathbf{Z}(n)\neq \mathbf{0}\right\} $ is, for a given
$n$, an inhomogeneous discrete-time Markov branching process in $m$.
\end{remark}

Observe that for $\mathbf{k}=\left( k_{0},k_{1}\right) $
\begin{equation}
\mathbf{E}\Big[\mathbf{s}^{\mathbf{Z}(k_{1},n)}|\mathbf{Z}%
(k_{0},n)=(r_{1},r_{2})\Big]=\left( \hat{J}_{1}^{\,(\mathbf{k}%
,n)}(s_{1})\right) ^{r_{1}}\left( \hat{J}_{2}^{\,(\mathbf{k}%
,n)}(s_{1},s_{2})\right) ^{r_{2}} . \label{decompoJ1J2}
\end{equation}%
Hence to investigate this quantity we can use the equalities
\begin{equation}
\hat{J}_{1}^{\,(\mathbf{k},n)}(s_{1})=1-\frac{%
Q_{1}(k_{1}-k_{0},1-(1-s_{1})Q_{1}(n-k_{1}))}{Q_{1}(n-k_{0})}
\label{expreZ1k1}
\end{equation}%
and
\begin{equation}
\hat{J}_{2}^{\,(\mathbf{k},n)}(\mathbf{s})=1-\frac{%
Q_{21}(k_{1}-k_{0};\mathbf{1}-(\mathbf{1}-\mathbf{s})\otimes \mathbf{Q}%
_{21}(n-k_{1}))}{Q_{21}(n-k_{0})}.  \label{forZ1k1Z2k2}
\end{equation}

We are now able to prove convergence of the finite-dimensional distributions
of the prelimiting processes to the corresponding finite-dimensional
distributions of the limiting processes in Theorems \ref{T_ZubQ1micro} to %
\ref{T_VatSagmicro}.

\subsection{Convergence of finite-dimensional distributions in Theorem
\protect\ref{T_ZubQ1micro}}
\hspace{1cm}\\
\noindent \textbf{Point (1):} For $t\geq 0$ introduce the function
\begin{eqnarray*}
f_{21}(t;s_{2}) =1-\Big(1-e^{-t}+e^{-t}(1-s_{2})^{-\alpha _{2}}\Big)^{-\frac{%
1}{\alpha _{2}}} =\lim_{n\rightarrow \infty }\mathbf{E}_{2}\left[ \mathbf{s}^{\mathbf{Z}((1-e^{-t})n,n)}|\mathbf{Z}(n)\neq \mathbf{%
0}\right] ,
\end{eqnarray*}%
whose first order derivative with respect to $t$ is
\begin{equation*}
\frac{\partial f_{21}(t;s_{2})}{\partial t}=\frac{1}{\alpha _{2}}\Big(%
e^{-t}-e^{-t}(1-s_{2})^{-\alpha _{2}}\Big)\Big(1-e^{-t}+e^{-t}(1-s_{2})^{-%
\alpha _{2}}\Big)^{-1-1/\alpha _{1}}.
\end{equation*}
Recalling the definition of $g_{2}^{(X)}$ in \eqref{Defg2} we get
\begin{multline*}
g_{2}^{(X)}(f_{21}(t;s_{2}))-f_{21}(t;s_{2}) =\frac{1}{\alpha _{2}}\Big(%
(1-f_{21}(t;s_{2}))^{1+\alpha _{2}}-1+(1+\alpha _{2})f_{21}(t;s_{2})\Big)%
-f_{21}(t;s_{2}) \\
=\frac{1}{\alpha _{2}}\Big((1-e^{-t}+e^{-t}(1-s_{2})^{-\alpha
_{2}})^{-1-1/\alpha _{2}}-(1-e^{-t}+e^{-t}(1-s_{2})^{-\alpha
_{2}})^{-1/\alpha _{2}}\Big) \\
=\frac{1}{\alpha _{2}}(1-e^{-t}+e^{-t}(1-s_{2})^{-\alpha _{2}})^{-1-1/\alpha
_{2}}\Big(1-(1-e^{-t}+e^{-t}(1-s_{2})^{-\alpha _{2}})\Big)=\frac{\partial
f_{21}(t;s_{2})}{\partial t}.
\end{multline*}
Applying Lemma \ref{lemmaathreyaney} we conclude that for $%
s_1,s_2 \in [0,1]$
\begin{equation*}
\underset{n \to \infty}{\lim}J_2^{(0,(1-e^{-t})n,n)}(\mathbf{s})=%
f_{21}(t;s_{2})=\mathbf{E}\left[ s_{2}^{X(t)}|X(0)=1\right],
\end{equation*}
where $X(\cdot )$ has been specified in Definition \ref{defprocessX}.
This, along with Theorem \ref{T_ZubQ1}(1) proves convergence of
one-dimensional distributions of the prelimiting process to that of $%
X(\cdot) $.

In view of Lemma \ref{L_convol}, to check the needed convergence of
arbitrary finite-dimensional distributions it is sufficient to consider the
case $p=2$. Let
\begin{equation*}
k_{1}=\left[ (1-e^{-t_{1}})n\right] <\left[ (1-e^{-t_{2}})n\right]=k_2 ,\quad 0<t_{1}<t_{2}.
\end{equation*}%
By the previous estimates, Lemma \ref{L_convol}, Theorem \ref{T_ZubQ1}(1)
and the uniform continuity in $t_{1}$ and $t_{2}$ of the functions involved,
we get
\begin{multline*}
\lim_{n\rightarrow \infty }\hat{J}_{2}^{\,(\mathbf{k},n)}(\mathbf{S}_{1},
\mathbf{S}_{2}) =\lim_{n\rightarrow \infty }J_{2}^{\,(\Delta _{1},n)}(
\mathbf{S}_{1}\otimes \mathbf{J}^{\left( \Delta _{2},n-k_{1}\right) }(
\mathbf{S}_{2})) =f_{21}\left( t_{1};s_{21}\lim_{n\rightarrow \infty
}J_{2}^{\left( \Delta _{2},n-k_{1}\right) }(\mathbf{S}_{2})\right) \\
=f_{21}\left( t_{1};s_{21}f_{21}\left( t_{2}-t_{1};s_{22}\right) \right) =%
\mathbf{E}\left[ s_{21}^{X(t_{1})}s_{22}^{X(t_{2})}|X(0)=1\right]
\end{multline*}
as desired.

\noindent \textbf{Point (2):} Convergence of one-dimensional distributions
has been established in Theorem \ref{T_ZubQ1}(2). In particular,
there is no type $1$ individuals in the limit process. As before, to check
the needed convergence of arbitrary finite-dimensional distributions it is
sufficient to consider the case $p=2$ and to calculate the transition
density of the limiting process for type $2$ individuals. Let $h
$ be an integer valued function such that $h^{\ast }(n)\ll h(n)\ll n$, and $%
k_{0}=n-t_{0}h(n)<k_{1}=n-t_{1}h(n)$. Take $s_2$ satisfying, for a positive $%
\lambda $
\begin{equation*}
1-s_{2}=1-\exp \left\{ -\lambda \frac{Q_{2}(n)}{Q_{21}(t_{2}h(n))%
}\right\} \sim \lambda \frac{Q_{2}(n)}{Q_{21}(t_{2}h(n))}.
\end{equation*}%
Let, further,
\begin{equation*}
Z_{2}(k_{1},n)=y\frac{Q_{21}(t_{1}h(n))}{Q_{2}(n)}
\end{equation*}%
(recall that we agree to consider the equalities of the form $Z\left(
x,n\right) =y$ as $Z\left( \left[ x\right] ,n\right) =\left[ y\right] $).
Then from
\eqref{expreZ1k1} and the fact that there is no type
$1$ individuals we get for $s_1 $ in $[0,1]$,
\begin{equation*}
1-\hat{J}_{2}^{\,(\mathbf{k},n)}(\mathbf{s})\sim \frac{1-\mathbf{%
E}_{2}\Big[e^{-\lambda Q_{2}(n)Z_{2}((t_{0}-t_{1})h(n))}\Big]}{%
Q_{21}(t_{0}h(n))}\sim \frac{\lambda Q_{2}(n)}{Q_{21}(t_{0}h(n))},
\end{equation*}%
where we applied Lemma \ref{L_trick}. Hence we deduce that
\begin{equation*}
\log \mathbf{E}\left[ \mathbf{s}^{\mathbf{Z}(k_{1},n)}|\mathbf{Z}%
(k_{0},n)=y\frac{Q_{21}(t_{0}h(n))}{Q_{2}(n)}\mathbf{e}_{2}\right] =y\frac{%
Q_{21}(t_{0}h(n))}{Q_{2}(n)}\log \hat{J}_{2}^{\,(\mathbf{k},n)}(\mathbf{s}%
)\sim -\lambda y.
\end{equation*}

\noindent \textbf{Point (3):} Once again, it is sufficient to calculate the
transition density of the limit process only. Let $k_{0}=n-t_{0}h^{\ast
}(n)<k_{1}=n-t_{1}h^{\ast }(n)$, where $h^{\ast }(n)$ has been defined in (%
\ref{hstar}). Applying first \eqref{expreZ1k1} we get
\begin{eqnarray*}
1-\hat{J}_{1}^{(\mathbf{k},n)}(s_{1})&\sim &\frac{1-\mathbf{E}_{1}\Big[ %
e^{-(1-s_{1})Q_{1}(t_1h^{\ast }(n))Z_{1}((t_{0}-t_{1})h^{\ast
}(n))}\Big]}{ Q_{1}(t_{0}h^{\ast }(n))} \\
&\sim &\frac{1-\mathbf{E}_{1}\Big[e^{-(1-s_{1})((t_{0}-t_{1})/t_1%
)^{1/\alpha _{1}}Q_{1}(\left( t_{0}-t_{1}\right) h^{\ast
}(n))Z_{1}((t_{0}-t_{1})h^{\ast }(n))}\Big]}{((t_{0}-t_{1})/t_{0})^{1/\alpha
_{1}}Q_{1}((t_{0}-t_{1})h^{\ast }(n))} \\
&\sim &\Big(\frac{t_{0}}{t_{0}-t_{1}}\Big)^{\frac{1}{\alpha _{1}}}\Big[1+%
\frac{ t_1}{t_{0}-t_{1}}(1-s_{1})^{-\alpha _{1}}\Big]^{-\frac{1%
}{\alpha _{1}}} =\left( 1-\frac{t_{1}}{t_{0}}+\frac{t_{1}}{t_{0}}%
(1-s_{1})^{-\alpha _{1}}\right) ^{-\frac{1}{\alpha _{1}}}.
\end{eqnarray*}

Introduce now
\begin{equation*}
s_{2}=\exp \left\{ -\lambda _{2}\frac{Q_{2}(n)}{Q_{21}(t_{1}h^{\ast }(n))}%
\right\}.
\end{equation*}
Then, by Markov property and \eqref{expressreduced} we get
\begin{eqnarray*}
1-\hat{J}_{2}^{(\mathbf{k},n)}(\mathbf{s}) &=&1-\mathbf{E}_{2}%
\left[ \mathbf{s}^{\mathbf{Z}(\left( t_{0}-t_{1}\right) h^{\ast
}(n),t_{0}h^{\ast }(n))}|\mathbf{Z}\left( t_{0}h^{\ast }(n)\right) \neq
\mathbf{0}\right] \\
&=&\frac{Q_{21}(\left( t_{0}-t_{1}\right) h^{\ast }(n);\mathbf{1}-(\mathbf{1}%
-\mathbf{s})\otimes \mathbf{Q}_{21}(t_{1}h^{\ast }(n)))}{Q_{21}(t_{0}h^{\ast
}(n))} \\
&\sim& \frac{Q_{21}(\left( t_{0}-t_{1}\right) h^{\ast
}(n);1-(1-s_{1})t_{1}^{-1/\alpha _{1}}Q_{1}(h^{\ast }(n)),1-\lambda
_{2}Q_{2}(n))}{Q_{21}(t_{0}h^{\ast }(n))} \\
&\sim & \frac{Q_2(n)}{Q_{21}(t_{0}h^{\ast}(n))}\Big[ %
(1-s_{1})^{1-\alpha _{1}}t_{1}^{1-\frac{1}{\alpha_{1}}}\Big( 1-\Big( 1+\Big(
\frac{t_{0}}{t_{1}}-1\Big) (1-s_{1})^{\alpha _{1}}\Big) ^{1-\frac{1}{\alpha}
_{1}}\Big) +\lambda_2 \Big],
\end{eqnarray*}
where for the last equivalence we applied \eqref{C_zub2} with $%
a=t_{0}-t_{1}$ and $\lambda _{1}=(1-s_{1})t_{1}^{-1/\alpha _{1}}$. Hence,
\begin{multline*}
\mathbf{E}\left[ s_{1}^{Z_{1}(k_{1},n)}e^{-\lambda _{2}\frac{Q_{2}(n)}{%
Q_{21}(x_{1}h^{\ast }(n))}Z_{2}(k_{1},n)}|\mathbf{Z}(k_{0},n)=y\frac{%
Q_{21}(x_{0}h^{\ast }(n))}{Q_{2}(n)}\mathbf{e}_{2}\right] \\
\sim \exp \left\{ -\left( (1-s_{1})^{1-\alpha _{1}}t_{1}^{1-1/\alpha
_{1}}\left( 1-\left( 1+\left( \frac{t_{0}}{t_{1}}-1\right) (1-s_{1})^{\alpha
_{1}}\right) ^{1-1/\alpha _{1}}\right) +\lambda _{2}\right) y\right\} .
\end{multline*}
This proves the desired convergence of the finite-dimensional distribution
of the prelimiting process to those of $\mathbf{G(\cdot )}$.

\noindent \textbf{Point (4):} Let $h$ be an integer valued function such
that $1\ll h(n)\ll h^{\ast }(n)$, and $\lambda _{1},\lambda _{2}$ be
positive numbers. Selecting $k_{0}=n-t_{0}h(n)<n-t_{1}h(n)=k_{1}
$, we get by using \eqref{forZ1k1Z2k2}
\begin{multline*}
1-\hat{J}_{2}^{(k_{0},k_{1},n)}\Big(e^{-\lambda _{1}\frac{Q_{1}(h^{\ast }(n))%
}{Q_{1}(t_{1}h(n))}},e^{-\lambda _{2}\frac{Q_{2}(n)}{Q_{21}(t_{1}h(n))}}\Big)%
\sim \frac{1-\mathbf{E}_{2}\Big[\exp \left\{ -\lambda
_{2}Q_{2}(n)Z_{2}((t_{0}-t_{1})h(n))\right\} \Big]}{Q_{21}(t_{0}h(n))} \\
+\frac{\mathbf{E}_{21}\Big[e^{-\lambda _{2}Q_{2}(n)Z_{2}((t_{0}-t_{1})h(n))}%
\Big(1-e^{-\lambda _{1}Q_{1}(h^{\ast }(n))Z_{1}((t_{0}-t_{1})h(n))}\Big)\Big]%
}{Q_{21}(t_{0}h(n))}.
\end{multline*}%
By applying twice Lemma \ref{L_trick} as $h(n)\ll h^{\ast }(n)$ and $h(n)\ll
n$ we get that the first summand is equivalent to
\begin{equation*}
\frac{\lambda _{2}Q_{2}(n)}{Q_{21}(t_{0}h(n))},
\end{equation*}%
and the second summand is smaller than a term equivalent to
\begin{equation*}
\frac{\lambda _{1}Q_{1}(h^{\ast }(n))}{Q_{21}(t_{0}h(n))}=o\Big(\frac{%
Q_{2}(n)}{Q_{21}(t_{0}h(n))}\Big).
\end{equation*}%
This yields
\begin{equation*}
\mathbf{E}_{2}\Big[\exp \left\{ -\lambda _{1}\frac{Z_{1}(k_{1},n)Q_{1}(h^{%
\ast }(n))}{Q_{1}(t_{1}h(n))}-\lambda _{2}\frac{Z_{2}(k_{1},n)Q_{2}(n)}{%
Q_{21}(t_{1}h(n))}\right\} |\mathbf{Z}(k_{0},n)=y_{2}\frac{Q_{21}(t_{0}h(n))%
}{Q_{2}(n)}\mathbf{e}_{2}\Big]\sim e^{-\lambda _{2}y_{2}}.
\end{equation*}%
In the same way we deduce that
\begin{equation*}
1-\mathbf{E}_{2}\Big[\exp \left\{ -\lambda _{1}\frac{Z_{1}(k_{1},n)Q_{1}(h^{%
\ast }(n))}{Q_{1}(t_{1}h(n))}\right\} |\mathbf{Z}(k_{0},n)=y_{1}\frac{%
Q_{1}(t_{0}h(n))}{Q_{1}(h^{\ast }(n))}\mathbf{e}_{1}\Big]\sim e^{-\lambda
_{1}y_{1}}.
\end{equation*}%
The two last equivalences justify the claimed convergence of
finite-dimensional distributions.

\subsection{Convergence of finite-dimensional distributions in Theorem
\protect\ref{T_zubQ3micro}}

\hspace{2cm}\newline
\textbf{Point (1):} Applying Theorem \ref{T_zubQ3}(1) we can consider the function
\begin{equation*}
f_{21}(t;\mathbf{s}):=\lim_{n\rightarrow \infty }\mathbf{E}_{2}%
\left[ \mathbf{s}^{\mathbf{Z}(tg^{\ast }(n),n)}|\mathbf{Z}(n)\neq \mathbf{0}%
\right] =1-t^{-\frac{1}{\alpha _{2}}}\phi \Big((1-s_{1})\frac{t^{\frac{1}{%
\alpha _{2}}+1}}{\alpha _{2}A_{21}},(1-s_{2})t^{%
\frac{1}{\alpha _{2}}}\Big),
\end{equation*}%
whose first order derivative with respect to $t$ is
\begin{eqnarray*}
\frac{\partial f_{21}(t;\mathbf{s})}{\partial t} &=&\frac{1}{%
\alpha _{2}}t^{-\frac{1}{\alpha _{2}}-1}\Big(\phi (\tilde{\lambda}_{1},%
\tilde{\lambda}_{2})-(1+\alpha _{2})\tilde{\lambda}_{1}\frac{\partial \phi }{%
\partial \lambda _{1}}(\tilde{\lambda}_{1},\tilde{\lambda}_{2})-\tilde{%
\lambda}_{2}\frac{\partial \phi }{\partial \lambda _{2}}(\tilde{\lambda}_{1},%
\tilde{\lambda}_{2})\Big) \\
&=&\frac{1}{\alpha _{2}}\Big(t^{-\frac{1}{\alpha _{2}}-1}\phi
^{1+\alpha _{2}}(\tilde{\lambda}_{1},\tilde{\lambda}_{2})-1+s_{1}\Big),
\end{eqnarray*}%
where we have applied \eqref{defphi} and used the notations
\begin{equation*}
\tilde{\lambda}_{1}=(1-s_{1})t^{1/\alpha _{2}+1}/(\alpha _{2}A_{21})\quad
\text{and}\quad \tilde{\lambda}_{2}=(1-s_{2})t^{\frac{1}{\alpha _{2}}}.
\end{equation*}%
Now taking $f_{1}(t;s_{1})=s_{1}$ and recalling the definition of $%
g_{21}^{(Y)}$ in \eqref{pgfeq1} we obtain
\begin{equation*}
g_{21}^{(Y)}(s_{1},f_{21}(t;\mathbf{s}))-f_{21}(t;
\mathbf{s})=\frac{\alpha _{2}}{1+\alpha _{2}}\frac{\partial f_{21}(t;\mathbf{s})}{\partial t}.
\end{equation*}%
Applying Lemma \ref{lemmaathreyaney} yields that $f_{21}$ is
the generating function of the process $\mathbf{Y}$ specified in Definition %
\ref{defprocessX}. Hence the desired convergence of  finite-dimensional
distributions follows by the same arguments as before.

\noindent \textbf{Point (2):} We take $g^{\ast }(n)\ll k_{0}<k_{1}\ll n$. From Theorem \ref{T_zubQ3}(2) we know that there is no more type $%
2$ individuals in the limit process. Applying \eqref{expreZ1k1} and Lemma %
\ref{L_trick} we get
\begin{equation*}
\hat{J}_1^{(\,\mathbf{k},n)}(s_1)=1-\frac{%
Q_{1}(k_{1}-k_{0},1-(1-s_{1})Q_{1}(n-k_{1}))}{Q_{1}(n-k_{0})} \sim 1- \frac{(1-s_{1})Q_{1}(n-k_{1})}{Q_{1}(n-k_{0})}\sim s_{1}.
\end{equation*}

\noindent \textbf{Point (3):} Again as there is no type $2$
individuals it is sufficient to study the dynamics of type $1$ individuals.
Applying Theorem \ref{T_zubQ3}(3) we can consider the function
\begin{equation*}
f_{1}(t;s_{1})=\lim_{n\rightarrow \infty }\mathbf{E}_{2}\left[ \mathbf{s}^{%
\mathbf{Z}((1-e^{-t})n,n)}|\mathbf{Z}(n)\neq \mathbf{0}\right]
=1-((1-e^{-t})+e^{-t}(1-s_{1})^{-\alpha _{1}})^{-1/\alpha _{1}(1+\alpha
_{2})},
\end{equation*}%
whose first order derivative with respect to $t$ is
\begin{equation*}
\frac{\partial f_{1}(t;s_{1})}{\partial t}=\frac{1}{\alpha _{1}(1+\alpha
_{2})}\left( e^{-t}-e^{-t}(1-s_{1})^{-\alpha _{1}}\right) \left(
1-e^{-t}+e^{-t}(1-s_{1})^{-\alpha _{1}}\right) ^{-1-1/(\alpha _{1}(1+\alpha
_{2}))}.
\end{equation*}%
Recalling the definition of $g_{1}^{(V)}$ in \eqref{pgfeq4} we obtain
\begin{equation*}
g_{1}^{(V)}(f_{1}(t;s_{1}))-f_{1}(t;s_{1})=\frac{1}{\alpha _{1}
(1+\alpha _{2})}\Big((1-f_{1}(t;s_{1}))^{1+\alpha _{1}(1+\alpha
_{2})}-1+f_{1}(t;s_{1})\Big)=\frac{\partial f_{1}(t;s_{1})}{\partial t}.
\end{equation*}%
Applying Lemma \ref{lemmaathreyaney} yields that $f_{1}$ is the generating
function of the process $V(\cdot )$ specified in Definition \ref{defprocessX}%
(3). Hence the desired convergence of finite-dimensional distributions
follows.

\noindent \textbf{Point (4):} Let $h$ be an integer-valued function
satisfying $1\ll h(n)\ll n$, and $k_{0}=n-t_{0}h(n)<k_{1}=n-t_{1}h(n)$. The one dimensional distribution of the limit process is given by
Theorem \ref{T_zubQ3}(4). Applying Lemma \ref{L_trick} and %
\eqref{expreZ1k1} yields
\begin{multline*}
1-\mathbf{E}\Big[e^{-\lambda _{1}\frac{Q_{1}(n)}{Q_{1}(t_{1}h(n))}%
Z_{1}(k_{1},n)}|\mathbf{Z}(k_{0},n)=\mathbf{e}_{1}\Big] \sim \frac{1-\mathbf{%
E}_{1}\Big[e^{-\lambda _{1}Q_{1}(n)Z_{1}((t_{0}-t_{1})h(n))}\Big]}{%
Q_{1}(t_{0}h(n))}\sim \lambda _{1}\frac{Q_{1}(n)}{Q_{1}(t_{0}h(n))}.
\end{multline*}%
Hence we deduce
\begin{equation*}
\Big[e^{-\lambda _{1}\frac{Q_{1}(n)}{Q_{1}(x_{1}h(n))}Z_{1}(k_{1},n)}|%
\mathbf{Z}(k_{0},n) =y\frac{Q_{1}(x_{0}h(n))}{Q_{1}(n)}\mathbf{e}_{1}\Big] %
\sim e^{-\lambda_1 y},
\end{equation*}
as required.

\subsection{Convergence of finite-dimensional distributions in Theorem
\protect\ref{T_VatSagmicro}\label{SubVatSagMicro}}

\hspace{1cm}\newline
\noindent \textbf{Point (1):} Consider the function
\begin{multline*}
1-f_{1}(t;s_{1})=1-\lim_{n\rightarrow \infty }\mathbf{E}_{1}\left[
s_{1}^{Z_{1}((1-e^{-t})n,n)}|\mathbf{Z}(n)\neq \mathbf{0}\right]  \\
=\lim_{k_{1}\rightarrow \infty }\frac{1-\mathbf{E}_{1}\Big[%
e^{-(1-s_{1})(e^{t}-1)^{1/\alpha _{1}}Q_{1}(k_{1})Z_{1}(k_{1})}\Big]}{%
(1-e^{-t})^{1/\alpha _{1}}Q_{1}(k_{1})}=\Big(1-e^{-t}+e^{-t}(1-s_{1})^{-%
\alpha _{1}}\Big)^{-\frac{1}{\alpha _{1}}},
\end{multline*}%
where we applied \eqref{expreZ1k1} with $k_{0}=0$ and $k_{1}=%
\left[ (1-e^{-t})n\right] $, and \eqref{monotype}. Then following the proof
of Theorem \ref{T_ZubQ1micro}(1) we obtain
\begin{equation*}
g_{1}^{(W)}(f_{1}(t;s_{1}))-f_{1}(t;s_{1})=\frac{\partial f_{1}(t;s_{1})}{%
\partial t}.
\end{equation*}%
Now applying Theorem \ref{T_VatSag}(1) we can consider the
function
\begin{equation*}
f_{21}(t;\mathbf{s}):=\lim_{n\rightarrow \infty }\mathbf{E}_{2}%
\left[ \mathbf{s}^{\mathbf{Z}((1-e^{-t})n,n)}|\mathbf{Z}(n)\neq
\mathbf{0}\right] =1-(1-e^{-t})^{-1/\alpha _{2}}\psi (\bar{%
\lambda}_{1},\bar{\lambda}_{2}),
\end{equation*}%
where, for the sake of readability we have
introduced
\begin{equation*}
\bar{\lambda}_{1}=(1-s_{1})(e^{t}-1)^{1/\alpha _{1}}\quad \text{and}\quad
\bar{\lambda}_{2}=(1-s_{2})(e^{t}-1)^{1/\alpha _{2}}.
\end{equation*}%
Since  $1+\alpha_{2}^{-1}=\alpha_{1}^{-1}$ under Assumption \eqref{VaSag}, we get
\begin{multline*}
\frac{\partial f_{21}(t;\mathbf{s})}{\partial t}%
=(1-e^{-t})^{-1/\alpha _{2}-1}\Big(\frac{e^{-t}}{\alpha _{2}}\psi (\bar{%
\lambda}_{1},\bar{\lambda}_{2})-\frac{\bar{\lambda}_{1}}{\alpha _{1}}\frac{%
\partial \psi }{\partial {\lambda }_{1}}(\bar{\lambda}_{1},\bar{\lambda}%
_{2})-\frac{\lambda _{2}}{\alpha _{2}}\frac{\partial \psi }{\partial \bar{%
\lambda}_{2}}(\bar{\lambda}_{1},\bar{\lambda}_{2})\Big) \\
=\frac{(1-e^{-t})^{-1/\alpha _{2}-1}}{\alpha _{2}}\Big((e^{-t}-1)\psi (\bar{%
\lambda}_{1},\bar{\lambda}_{2})+b^{\alpha _{2}}\psi ^{1+\alpha _{2}}(\bar{%
\lambda}_{1},\bar{\lambda}_{2})-\alpha _{2}A_{12}\bar{\lambda}%
_{1}\Big(1+\Big(\frac{b}{\sigma }\bar{\lambda}_{1}\Big)^{\alpha _{1}}\Big)%
^{-1/\alpha _{1}}\Big),
\end{multline*}%
where we have applied \eqref{Partial2}. Further, recalling the
definitions of $g_{21}^{(W)}$ and $\kappa $ in \eqref{pgfeqW}
and \eqref{defkappa} we see that
\begin{multline*}
g_{21}^{(W)}(f_{1}(t;s_{1}),f_{21}(t;\mathbf{s}))-f_{21}(t;\mathbf{s}) \\
=\frac{1}{\kappa}\Big(\frac{\sigma A_{12}}{b}f_{1}(t;s_{1})+%
\frac{b^{\alpha _{2}}}{\alpha _{2}}\Big((1-f_{21}(t;\mathbf{s}%
))^{1+\alpha _{2}}-1+(1+\alpha _{2})f_{21}(t;\mathbf{s})\Big)%
\Big)-f_{21}(t;\mathbf{s}).
\end{multline*}%
Observing that
\begin{equation*}
\frac{1}{\kappa }\frac{b^{\alpha _{2}}(1+\alpha _{2})}{\alpha
_{2}}-1=\frac{\alpha _{2}}{(1+\alpha _{2})b^{\alpha _{2}}-1}\frac{b^{\alpha
_{2}}(1+\alpha _{2})}{\alpha _{2}}-1=\frac{1}{(1+\alpha _{2})b^{\alpha
_{2}}-1}=\frac{1}{\alpha _{2}\kappa }
\end{equation*}%
and recalling the definition of $b$ in \eqref{defb} leads to
\begin{multline*}
g_{21}^{(W)}(f_{1}(t;s_{1}),f_{21}(t;\mathbf{s}))-f_{21}(t;\mathbf{s}) \\
=\frac{1}{\alpha _{2}\kappa }\Big(b^{\alpha _{2}}(1-f_{21}(t;\mathbf{s}))^{1+\alpha _{2}}-\frac{\sigma \alpha _{2}A_{12}}{b}%
(1-f_{1}(t;s_{1}))-1+f_{21}(t;\mathbf{s})\Big) \\
=\frac{(1-e^{-t})^{-1-1/\alpha _{2}}}{\alpha _{2}\kappa }\Big({%
b^{\alpha _{2}}}\psi ^{1+\alpha _{2}}(\bar{\lambda}_{1},\bar{\lambda}_{2})-%
\frac{\sigma \alpha _{2}A_{12}}{b}\Big(1+\frac{e^{-t}}{1-e^{-t}}%
(1-s_{1})^{-\alpha _{1}}\Big)^{-\frac{1}{\alpha _{1}}}-(1-e^{-t})\psi (\bar{%
\lambda}_{1},\bar{\lambda}_{2})\Big),
\end{multline*}%
where we have used again the equality $1+\alpha _{2}^{-1}=\alpha _{1}^{-1}$.
This, finally, yields
\begin{equation*}
g_{21}^{(W)}(f_{1}(t;s_{1}),f_{21}(t;\mathbf{s}))-f_{21}(t;\mathbf{s})=\frac{1}{\kappa }\frac{\partial
f_{21}(t;\mathbf{s})}{\partial t}.
\end{equation*}%
Applying Lemma \ref{lemmaathreyaney} we conclude that $f_{21}$ is the
generating function of the process $\mathbf{W}\left( \cdot \right) $
specified in Definition \ref{defprocessX}(4).

\noindent \textbf{Point (2):} Let $1\ll h(n)\ll n$ and $%
k_{0}=n-t_{0}h(n)<k_{1}=n-t_{1}h(n)$ for $0<t_{1}<t_{0}$. Choose
$\lambda_1, \lambda_2 \geq 0$.
Using \eqref{expreZ1k1}, the fact that $h(n)\gg 1$ and recalling Lemma \ref%
{L_trick} we see that
\begin{equation*}
1-\hat{J}_1^{\,(k_0,k_1,n)}\Big(e^{-\lambda_{1}\frac{Q_{1}(n)}{%
Q_{1}(x_{1}h(n))}}\Big) \sim \frac{Q_{1}((t_{0}-t_{1})h(n),1-\lambda
_{1}Q_{1}(n))}{Q_{1}(x_{0}h(n))}\sim \frac{\lambda _{1}Q_{1}(n)}{%
Q_{1}(x_{0}h(n))}.
\end{equation*}
Hence
\begin{eqnarray*}
\log \mathbf{E}\Big[e^{-\lambda _{1}\frac{Q_{1}(n)}{Q_{1}(t_{1}h(n))}%
Z_{1}(k_{1},n)}|\mathbf{Z}(k_{0},n) =y\frac{Q_{1}(t_{0}h(n))}{Q_{1}(n)}%
\mathbf{e}_{1}\Big] \sim -\lambda _{1}y.
\end{eqnarray*}%
Then from \eqref{forZ1k1Z2k2} \ we obtain
\begin{multline*}
1-\hat{J}_2^{\,(k_0,k_1,n)}\Big(e^{-\frac{\lambda_{1}Q_{1}(n)}{%
Q_{1}(t_{1}h(n))}} , e^{-\frac{\lambda_{2}Q_{2}(n)}{Q_{2}(t_{1}h(n))}} \Big)
\sim \frac{1-\mathbf{E}_{2}\Big[e^{-\lambda
_{1}Q_{1}(n)Z_{1}((t_{0}-t_{1})h(n))}e^{-\lambda
_{2}Q_{21}(n)Z_{2}((t_{0}-t_{1})h(n))}\Big]}{Q_{21}(t_{0}h(n))} \\
=\frac{1-\mathbf{E}_{2}\Big[e^{-\lambda
_{2}Q_{21}(n)Z_{2}((t_{0}-t_{1})h(n))}\Big]}{Q_{21}(t_{0}h(n))}+\frac{%
\mathbf{E}_{2}\Big[e^{-\lambda _{2}Q_{21}(n)Z_{2}((t_{0}-t_{1})h(n))}\Big(%
1-e^{-\lambda _{1}Q_{1}(n)Z_{1}((t_{0}-t_{1})h(n))}\Big)\Big]}{%
Q_{21}(t_{0}h(n))}.
\end{multline*}%
To control the first term we apply Lemma \ref{L_trick} and Equivalence %
\eqref{surviv}:
\begin{equation*}
\frac{1-\mathbf{E}_{2}\Big[e^{-\lambda
_{2}Q_{21}(n)Z_{2}((t_{0}-t_{1})h(n))} \Big]}{Q_{21}(t_{0}h(n))}\sim \frac{%
\lambda _{2}Q_{21}(n)}{Q_{21}(t_{0}h(n))}.
\end{equation*}%
The second term is non-negative and bounded by
\begin{equation*}
\frac{\mathbf{E}_{2}\Big[1-e^{-\lambda _{1}Q_{1}(n)Z_{1}((t_{0}-t_{1})h(n))}%
\Big]}{Q_{21}(t_{0}h(n))} \sim \frac{\lambda _{1}(t_{0}-t_{1})h(n)Q_{1}(n)}{%
Q_{21}(t_{0}h(n))}=o\Big(\frac{\lambda _{2}Q_{21}(n)}{Q_{21}(t_{0}h(n))}\Big),
\end{equation*}
where the equality is a direct consequence of Condition (\ref{VaSag}) and
Equivalence \eqref{surviv}. This ends the proof of point (2).

\section{Tightness\label{S_tightness}}

In this section we complete the proof of Theorems \ref{T_ZubQ1micro}-\ref%
{T_VatSagmicro} by checking tightness. A key tool will be a
slightly simplified version of Theorem 6.5.4 in \cite{GS71} giving a
convergence criterion in Skorokhod topology for a class of Markov processes
which we formulate to simplify references. Let $a_1 \leq a_2$ be
in $\mathbf{R}_+$ and
\begin{equation*}
\left\{ \mathbf{K}_{n}(u)=\left( K_{1n}(u),K_{2n}(u)\right) ,a_{1}\leq u\leq
a_{2}\right\} , \quad n=1,2,...
\end{equation*}
be a sequence of Markov processes with values in $\mathbf{R}^{2}$ whose
trajectories belong with probability 1 to the space $\mathbf{D}_{[a_{1},a_{2}]}(%
\mathbf{R}^{2})$ of c\`adl\`ag functions on $[a_{1},a_{2}]$.

\begin{theorem}
\label{T_skoroh}If the finite-dimensional distributions of $\left\{ \mathbf{K%
}_{n}(u),a_{1}\leq u\leq a_{2}\right\} $ converge, as $n\rightarrow \infty $
to the respective finite-dimensional distributions of a process $\left\{
\mathbf{K}(u),a_{1}\leq u\leq a_{2}\right\} $ and for any $\varepsilon >0$
\begin{equation*}
\lim_{g\downarrow 0}\overline{\lim_{n\rightarrow \infty }}\sup_{0\leq
u_{1}-u_{0}\leq g}\sup_{\mathbf{y\in }\mathbb{R}^{2}}\mathbf{P}\left(
\left\Vert \mathbf{K}_{n}(u_{1})-\mathbf{y}\right\Vert \mathbf{>}\varepsilon
|\,\mathbf{K}_{n}(u_{0})=\mathbf{y}\right) =0,
\end{equation*}%
then, as $n\rightarrow \infty $
\begin{equation*}
\mathcal{L}\left\{ \mathbf{K}_{n}(u),a_{1}\leq u\leq a_{2}\right\}
\Longrightarrow \mathcal{L}\left\{ \mathbf{K}(u),a_{1}\leq u\leq
a_{2}\right\} .
\end{equation*}
\end{theorem}

With this theorem in hands, the method of proving the desired tightness is,
essentially, the same for all three theorems and has been used in various
situations in \cite{fleischmann1977structure}, \cite{vatutin2014macroscopic}
and \cite{Vat15}. For instance, to prove the tightness in Theorems \ref%
{T_ZubQ1micro} and \ref{T_zubQ3micro} one can apply a combination of
arguments used in \cite{fleischmann1977structure} and \cite{vatutin2014macroscopic}. The main difference with
these previous proofs is related with the non finiteness of
offspring variances under Conditions \eqref{F1}-\eqref{defA21}. Since the
remaining parts of the proofs follow the reasoning used in \cite%
{fleischmann1977structure} and \cite{vatutin2014macroscopic} we check in
this section the tightness for Theorem \ref{T_VatSagmicro} only. In what
follows it will be convenient to write $\mathbf{P}^{(n)}(\mathcal{B})$ for $%
\mathbf{P}(\mathcal{B}|\mathbf{Z}(n)\neq \mathbf{0})$ for any admissible
event~$\mathcal{B}$ and to use the notation $\left\Vert \mathbf{x}%
\right\Vert =\left\vert x_{1}\right\vert +\left\vert x_{2}\right\vert $ for $%
\mathbf{x=}\left( x_{1},x_{2}\right) $.\newline

\noindent \textbf{Proof of tightness for Theorem \ref{T_VatSagmicro} (1)
}. We will use in this proof ideas from \cite{fleischmann1977structure}. As we have mentioned, according to Lemma %
\ref{L_convol} the law $\mathbf{P}(\left\{ \mathbf{Z}(m,n),0\leq
m\leq n\right\} \in (\cdot )|\mathbf{Z}(n)\neq \mathbf{0})$ specifies, for
each fixed $n$ an inhomogeneous Markov branching process. We denote its
transition probabilities by $\mathbf{P}^{(n)}(m_{1},\mathbf{z};m_{2},(\cdot
))$. In the case under consideration, $\mathbf{K}_{n}$ writes:
\begin{equation*}
\mathbf{K}_{n}(u)=\left\{ \mathbf{Z}(un,n)|\mathbf{Z}(n)\neq \mathbf{0}%
\right\} ,0\leq u<1.
\end{equation*}%
Hence, we will use Theorem \ref{T_skoroh} for each $\left[ 0,U%
\right] \subset \lbrack 0,1)$. Denote
\begin{equation*}
\mathcal{C}(k)=\left\{ \mathbf{z}\in \mathbb{Z}_{+}^{2}:\left\Vert \mathbf{z}%
\right\Vert \leq k\right\} ,
\end{equation*}%
take $0<u_{0}<u_{1}<1$, set
\begin{equation*}
m_{j}=u_{j}n=\left( 1-e^{-t_{j}}\right) n,\,\,j=0,1
\end{equation*}%
and, finally, let
\begin{equation*}
\mathcal{U}(g)=\left\{ \left( u_{1},u_{0}\right) \in \lbrack 0,U]\times
\lbrack 0,U]:0\leq u_{1}-u_{0}\leq g\right\} .
\end{equation*}

\begin{lemma}
\label{L_ceretain event.}Under the conditions of Theorem \ref{T_VatSagmicro}%
(1) for any fixed $k$ and $0<U<1$%
\begin{equation*}
\lim_{g\downarrow 0}\overline{\lim_{n\rightarrow \infty }}\sup_{\mathcal{U}%
(g)}\sup_{\mathbf{z}\in \mathcal{C}(k)}\mathbf{P}^{(n)}(\mathbf{Z}%
(m_{1};n)\neq \mathbf{z}|\mathbf{Z}(m_{0};n)=\mathbf{z})=0.
\end{equation*}
\end{lemma}

\begin{proof}
Let $n$ be fixed in $\mathbf{N}$. Notice that when a jump occurs
in the process $\{\mathbf{Z}(m,n),0\leq m\leq n\}$, the process $|\mathbf{Z}%
(.,n)|+Z_{1}(.,n)$ is incremented by one. Hence if $m_{0}\leq m_{1}\leq m_{2}
$, then $\mathbf{Z}(m_{0},n)\neq \mathbf{Z}(m_{1},n)$ implies $%
\mathbf{Z}(m_{0},n)\neq \mathbf{Z}(m_{2},n)$. Adding Markov property we get
\begin{multline*}
\sup_{\mathcal{U}(g)}\sup_{\mathbf{z}\in \mathcal{C}(k)}\mathbf{P}^{(n)}(%
\mathbf{Z}(m_{1};n)\neq \mathbf{z}|\mathbf{Z}(m_{0};n)=\mathbf{z}) \\
\leq \sup_{\mathcal{U}(g)}\sup_{\mathbf{z}\in \mathcal{C}(k)}\mathbf{P}%
^{(n)}(\mathbf{Z}(m_{0}+gn;n)\neq \mathbf{z}|\mathbf{Z}(m_{0};n)=\mathbf{z})
\\
=\sup_{u_{0}\in \lbrack 0,U]}\sup_{\mathbf{z}\in \mathcal{C}(k)}\mathbf{P}%
^{(n-m_{0})}(\mathbf{Z}(gn;n-m_{0})\neq \mathbf{z}|\mathbf{Z}(0;n-m_{0})=%
\mathbf{z}).
\end{multline*}%
According to Lemma \ref{L_convol} and (\ref{monotype})

\begin{multline*}
1-\lim_{n\rightarrow \infty }\mathbf{E}_{1}\left[ 
s_{1}^{Z_{1}(gn,n-m_{0})}|\mathbf{Z}(n-m_{0})\neq \mathbf{0}\right]
=\lim_{n\rightarrow \infty }\frac{Q_{1}(gn
;1-(1-s_{1})Q_{1}(n-m_{0}-gn))}{Q_{1}(n-m_{0})} \\
=(1-u_{0})^{1/\alpha _{1}}\Big(g+(1-s_{1})^{-\alpha
_{1}}(1-u_{0}-g)\Big)^{-1/\alpha _{1}}.
\end{multline*}%
In particular,
\begin{equation*}
\lim_{n\rightarrow \infty }\mathbf{P}^{(n-m_{0})}(0,\mathbf{e}%
_{1};gn,\left\{ \mathbf{e}_{1}\right\} )=1-g/(1-u_{0}).
\end{equation*}%
Therefore, for any $\varepsilon >0$ and $(u_{0},u_{1})\in
\mathcal{U}(g)$ there exists $n_{0}$ such that for all $n\geq n_{0}$
\begin{equation}
\lim_{n\rightarrow \infty }\mathbf{P}^{(n)}(m_{0},\mathbf{e}%
_{1};m_{1},\left\{ \mathbf{e}_{1}\right\} )\geq 1-g/(1-u_{0})-\varepsilon
\geq 1-g/(1-U)-\varepsilon =1-C_{1}g-\varepsilon ,  \label{ProbBelow1}
\end{equation}%
and
\begin{equation}
\lim_{g\downarrow 0}\lim_{n\rightarrow \infty }\sup_{\mathcal{U}(g)}\mathbf{P%
}^{(n)}(m_{0},\mathbf{e}_{1};m_{1},\left\{ \mathbf{e}_{1}\right\} )=1.
\label{tight1}
\end{equation}%
Similarly, by Lemma \ref{L_convol} and Theorem \ref{T_VatSag}(1) with $%
a=g/(1-u_{0})$ we get
\begin{eqnarray*}
&&\lim_{n\rightarrow \infty }\mathbf{E}\left[ \mathbf{s}^{\mathbf{Z}(m_{0}+gn,n)}|\mathbf{Z}(m_{0},n)=\mathbf{e}_{2}\right]
=\lim_{n\rightarrow \infty }\mathbf{E}_{2}\left[ \mathbf{s}^{\mathbf{Z}(gn,n-m_{0})}|\mathbf{Z}(n-m_{0})\neq \mathbf{0}\right]  \\
&&\qquad \qquad \qquad =1-a^{-1/\alpha _{2}}\psi \left( (1-s_{1})
\frac{\sigma }{b}\left( \frac{a}{1-a}\right) ^{1/\alpha
_{1}},(1-s_{2})\left( \frac{a}{1-a}\right) ^{1/\alpha _{2}}\right) .
\end{eqnarray*}%
Further, taking into account the boundary conditions (\ref{Inpsi}) and that $%
\alpha _{1}^{-1}=1+\alpha _{2}^{-1}$ yields
\begin{equation*}
\lim_{a\downarrow 0}a^{-1/\alpha _{2}}\psi \left( (1-s_{1})\frac{%
\sigma }{b}\left( \frac{a}{1-a}\right) ^{1/\alpha _{1}},(1-s_{2})\left(
\frac{a}{1-a}\right) ^{1/\alpha _{2}}\right) =1-s_{2}.
\end{equation*}%
Besides, for
\begin{equation*}
\lambda _{1}^{\ast }=\frac{\sigma }{b}\left( \frac{a}{1-a}%
\right) ^{1/\alpha _{1}}\quad \text{and}\quad \lambda _{2}^{\ast }=\left(
\frac{a}{1-a}\right) ^{1/\alpha _{2}}
\end{equation*}%
we have%
\begin{equation*}
\lim_{n\rightarrow \infty }\mathbf{P}^{(n-m_{0})}(0,\mathbf{e}%
_{2};gn,\left\{ \mathbf{e}_{2}\right\} )=\left( 1-a\right) ^{-1/\alpha _{2}}%
\frac{\partial \psi \left( \lambda _{1},\lambda _{2}\right) }{\partial
\lambda _{2}}\left\vert _{\left( \lambda _{1},\lambda _{2}\right) =\left(
\lambda _{1}^{\ast },\lambda _{2}^{\ast }\right) }\right. .
\end{equation*}%
Note that in view of (\ref{Inpsi})
\begin{equation*}
\frac{\partial \psi \left( \lambda _{1},\lambda _{2}\right) }{\partial
\lambda _{2}}\left\vert _{\left( \lambda _{1},\lambda _{2}\right) =\left(
\lambda _{1}^{\ast },\lambda _{2}^{\ast }\right) }\right. \geq 1-C_{2}(g),
\end{equation*}%
where $C_{2}(g)\rightarrow 0$ as $g\downarrow 0$. This implies, in
particular, that for any $\varepsilon >0$ there exists $n_{0}$ such that,
for all $n\geq n_{0}$
\begin{equation}
\mathbf{P}^{(n)}(m_{0},\mathbf{e}_{2};m_{1},\left\{ \mathbf{e}_{2}\right\}
)\geq \left( 1-a\right) ^{-1/\alpha _{2}}(1-C_{2}(g))-\varepsilon
=1-C_{3}(g)-\varepsilon   \label{ProbBelow2}
\end{equation}%
where $C_{3}(g)\rightarrow 0$ as $g\downarrow 0$. In particular,%
\begin{equation}
\lim_{g\downarrow 0}\lim_{n\rightarrow \infty }\sup_{\mathcal{U}(g)}\mathbf{P%
}^{(n)}(m_{0},\mathbf{e}_{2};m_{1},\left\{ \mathbf{e}_{2}\right\} )=1.
\label{tight2}
\end{equation}%
Using (\ref{tight1})-(\ref{tight2}), the branching property, the
decomposability of the process and the positivity of the offspring number of
each particle of the reduced process, we have for all $m_{1}=u_{1}n\geq
m_{0}=u_{0}n$ and $\mathbf{z}\in \mathcal{C}(k)$,
\begin{eqnarray*}
\lim_{g\downarrow 0}\liminf_{n\rightarrow \infty }\inf_{\mathcal{U}(g)}%
\mathbf{P}^{(n)}(m_{0},\mathbf{z};m_{1},\left\{ \mathbf{z}\right\} )
&=&\lim_{u_{1}\downarrow u_{0}}\liminf_{n\rightarrow \infty }\inf_{\mathcal{U%
}(g)}\prod_{j=1}^{2}(\mathbf{P}^{(n)}(m_{0},\mathbf{e}_{j};m_{1},\left\{
\mathbf{e}_{j}\right\} ))^{z_{j}} \\
&\geq &\lim_{g\downarrow 0}\liminf_{n\rightarrow \infty }\inf_{\mathcal{U}%
(g)}\prod_{j=1}^{2}(\mathbf{P}^{(n)}(m_{0},\mathbf{e}_{j};m_{1},\left\{
\mathbf{e}_{j}\right\} ))^{k}=1.
\end{eqnarray*}%
This implies the claim of the lemma.
\end{proof}

\begin{lemma}
\label{L_ratio}Let the conditions of Theorem \ref{T_VatSagmicro}(1) be
valid. If $m_{j}=u_{j}n,~j=0,1,2$ and $0\leq u_{0}<u_{1}<u_{2}<1$ with $%
u_{1}-u_{0}\leq g$, then there exists $C_{4}(g)\rightarrow 0$ as $%
g\downarrow 0$ such that for any $\varepsilon >0$ there exists $n_{0}$ such
that for all $n\geq n_{0}$ and $\mathbf{z}\in \mathcal{C}(k)$,
\begin{eqnarray*}
\mathbf{P}^{(n)}(\mathbf{Z}(m_{1},n)=\mathbf{z\,}|\,\mathbf{Z}(m_{0},n)=%
\mathbf{z};\left\Vert \mathbf{Z}(m_{2},n)\right\Vert \leq k) \geq \frac{%
\mathbf{P}^{(n)}(m_{0},\mathbf{z};m_{1},\left\{ \mathbf{z}\right\} )\mathbf{P%
}^{(n)}(m_{1},\mathbf{z};m_{2},\mathcal{C}(k))}{\mathbf{P}^{(n)}(m_{1},%
\mathbf{z};m_{2},\mathcal{C}(k))\mathbf{+}\left( C_{4}(g)+2\varepsilon
\right) k}.
\end{eqnarray*}
\end{lemma}

\begin{proof}
We have
\begin{eqnarray*}
\mathbf{P}^{(n)}(\mathbf{Z}(m_{1},n) =\mathbf{z\,}|\mathbf{Z}(m_{0},n)=%
\mathbf{z},\left\Vert \mathbf{Z}(m_{2},n)\right\Vert \leq k) =\mathbf{P}%
^{(n)}(m_{0},\mathbf{z};m_{1},\left\{ \mathbf{z}\right\} )\frac{\mathbf{P}%
^{(n)}(m_{1},\mathbf{z};m_{2},\mathcal{C}(k))}{\mathbf{P}^{(n)}(m_{0},%
\mathbf{z};m_{2},\mathcal{C}(k))}.
\end{eqnarray*}%
In view of (\ref{ProbBelow1})-(\ref{ProbBelow2}), there exists $%
n_0$ such that for $n \geq n_0$,
\begin{eqnarray*}
\mathbf{P}^{(n)}(m_{0},\mathbf{z};m_{2},\mathcal{C}(k)) &=&\sum_{\mathbf{z}%
^{\prime }}\mathbf{P}^{(n)}(m_{0},\mathbf{z};m_{1},\left\{ \mathbf{z}%
^{\prime }\right\} )\mathbf{P}^{(n)}(m_{1},\mathbf{z}^{\prime };m_{2},%
\mathcal{C}(k)) \\
&\leq &1-\mathbf{P}^{(n)}(m_{0},\mathbf{z};m_{1},\left\{ \mathbf{z}\right\}
)+\mathbf{P}^{(n)}(m_{1},\mathbf{z};m_{2},\mathcal{C}(k)) \\
&\leq &1-\left( 1-C_{3}(g)-\varepsilon \right) ^{k}\left(
1-C_{1}g-\varepsilon \right) ^{k}+\mathbf{P}^{(n)}(m_{1},\mathbf{z};m_{2},%
\mathcal{C}(k)) \\
&\leq &\left(C_1 g+C_{3}(g)+2\varepsilon \right) k+\mathbf{P}%
^{(n)}(m_{1},\mathbf{z};m_{2},\mathcal{C}(k)).
\end{eqnarray*}%
Hence the needed statement follows.
\end{proof}

\hspace{1cm} \newline

Let $U\in (0,1)$ and $\delta \in (0,U)$ be fixed and, in addition, $U+\delta
<1$.

\begin{lemma}
\label{L_Kskrokh}Under the conditions of Theorem \ref{T_VatSagmicro}(1) for
any fixed $k$ and $m_{0}=u_{0}n,\ m_{1}=u_{1}n,\ m_{2}=\left( U+\delta
\right) n$ we have%
\begin{equation*}
\lim_{g\downarrow 0}\overline{\lim_{n\rightarrow \infty }}\sup_{\mathcal{U}%
(g)}\sup_{\mathbf{z}\in \mathcal{C}(k)}\mathbf{P}^{(n)}(\mathbf{Z}%
(m_{1};n)\neq \mathbf{z\,}|\,\mathbf{Z}(m_{0};n)=\mathbf{z},\left\Vert
\mathbf{Z}(m_{2},n)\right\Vert \leq k)=0.
\end{equation*}
\end{lemma}

\begin{proof}
Let us first notice that from \eqref{decompoJ1J2} we can deduce
that if $\mathbf{w}=(w_{1},w_{2})\leq \mathbf{z}=(z_{1},z_{2})$ (where the
inequality is understood componentwise) then,
\begin{multline}  \label{decrease}
\mathbf{P}^{(n)}(m_{0},\mathbf{w};m_{1},\left\{ \mathbf{w}\right\} ) = \left( \frac{\hat{J}_{1}^{(m_0,m_1 ,n)}(s_{1})}{\partial s_1}%
\left\vert_{s_1=0}\right.\right) ^{w_{1}} \left( \frac{\hat{J}_{2}^{(m_0,m_1
,n)}(\mathbf{s})}{\partial s_2}\left\vert_{\mathbf{s}=\mathbf{0}%
}\right.\right) ^{w_{2}} \\
\geq \left( \frac{\hat{J}_{1}^{(m_0,m_1 ,n)}(s_{1})}{\partial s_1%
}\left\vert_{s_1=0}\right.\right) ^{z_{1}} \left( \frac{\hat{J}%
_{2}^{(m_0,m_1 ,n)}(\mathbf{s})}{\partial s_2}\left\vert_{%
\mathbf{s}=\mathbf{0}}\right.\right) ^{z_{2}} = \mathbf{P}^{(n)}(m_{0},%
\mathbf{z};m_{1},\left\{ \mathbf{z}\right\} ).
\end{multline}
By Lemma \ref{L_ratio}, Equations (\ref{ProbBelow1}) and (\ref%
{ProbBelow2}), for $m_{0}=u_{0}n$, $m_{1}=u_{1}n$ and $\mathbf{%
z} \in \mathcal{C}(k)$
\begin{multline*}
\mathbf{P}^{(n)}(\mathbf{Z}(m_{1},n)=\mathbf{z}|\mathbf{Z}(m_{0},n)=\mathbf{z%
};\left\Vert \mathbf{Z}(m_{2},n)\right\Vert \leq k) \\
\geq \left( 1-C_{2}(g)-\varepsilon \right) ^{k}\left( 1-C_{1}g-\varepsilon
\right) ^{k}\frac{\mathbf{P}^{(n)}(m_{1},\mathbf{z};m_{2},\mathcal{C}(k))}{%
\mathbf{P}^{(n)}(m_{1},\mathbf{z};m_{2},\mathcal{C}(k))\mathbf{+}\left(
C_{4}(g)+2\varepsilon \right) k}.
\end{multline*}
Using the decomposability hypothesis, the Markov property of the reduced
processes, the inequality $m_{2}-m_{1}\leq \left( U+\delta \right) n$ and \eqref{decrease} we obtain
\begin{multline*}
\mathbf{P}^{(n)}(m_{1},\mathbf{z};m_{2},\mathcal{C}(k))\geq \mathbf{P}%
^{(n)}(m_{1},\mathbf{z};m_{2},\left\{ \mathbf{z}\right\} ) =\prod_{j=1}^{2}(%
\mathbf{P}^{(n)}(m_{1},\mathbf{e}_{j};m_{2},\left\{ \mathbf{e}_{j}\right\}
))^{z_{j}} \\
\geq \prod_{j=1}^{2}(\mathbf{P}^{(n)}(m_{1},\mathbf{e}_{j};m_{2},\left\{
\mathbf{e}_{j}\right\} ))^{k} \geq \prod_{j=1}^{2}\left( \mathbf{P}^{(n)}(0,%
\mathbf{e}_{j};\left( U+\delta \right) n,\left\{ \mathbf{e}_{j}\right\}
)\right) ^{k}.
\end{multline*}%
It follows from Subsection \ref{SubVatSagMicro} that
\begin{equation*}
\lim_{n\rightarrow \infty }\prod_{j=1}^{2}\left( \mathbf{P}^{(n)}(0,\mathbf{e%
}_{j};\left( U+\delta \right) n,\left\{ \mathbf{e}_{j}\right\} )\right)
^{k}=\prod_{j=1}^{2}\mathbf{P}^{k}(\mathbf{W}(-\log(1-U-\delta)
)=\mathbf{e}_{j}|\mathbf{W}(0)=\mathbf{e}_{j})=B>0.
\end{equation*}%
Hence we get%
\begin{multline*}
\varliminf_{n\rightarrow \infty }\inf_{\mathcal{U}(g)}\inf_{\mathbf{z}\in
\mathcal{C}(k)}\mathbf{P}_{n}(\mathbf{Z}(m_{1};n)=\mathbf{z}|\mathbf{Z}%
(m_{0};n)=\mathbf{z},\left\Vert \mathbf{Z}(m_{2},n)\right\Vert \leq k) \\
\geq \left( 1-C_{2}(g)-\varepsilon \right) ^{k}\left( 1-C_{1}g-\varepsilon
\right) ^{k}\frac{B}{B\mathbf{+}\left( C_{4}(g)+2\varepsilon \right) k}.
\end{multline*}%
Letting first $g\downarrow 0$ and then $\varepsilon \downarrow 0$ completes
the proof of the lemma.
\end{proof}


Combining the statement of Lemma \ref{L_Kskrokh} and Theorem \ref%
{T_skoroh} about tightness and the finite dimensional convergence
established in Section \ref{SubVatSagMicro} point (1), we obtain
the following statement:

\begin{corollary}
\label{C_skk}Under the conditions of Lemma \ref{L_Kskrokh}
\begin{eqnarray*}
&&\mathcal{L}\left\{ \mathbf{Z}(un,n),0\leq u\leq U\,\Big|\,\left\Vert
\mathbf{Z}(m_{2},n)\right\Vert \leq k,\mathbf{Z}(n)\neq \mathbf{0}\right\} \\
&&\qquad \qquad \qquad \qquad \qquad \quad \,\Longrightarrow \mathcal{L}%
_{(0,1)}\left\{ \mathbf{W}(u),0\leq u\leq U|\,\left\Vert \mathbf{W}(U+\delta
)\right\Vert \leq k\right\} .
\end{eqnarray*}
\end{corollary}

\noindent \textbf{Final steps in proving Theorem \ref{T_VatSagmicro}(1):}
Let for $0<U<U_{1}=U+\delta <1$
\begin{equation*}
\begin{array}{lll}
\mathbf{P}^{(n)}(U;(\cdot )) & = & \mathbf{P}^{(n)}(\left\{ \mathbf{Z}%
(un,n),0\leq u\leq U\right\} \in (\cdot )), \\
\mathbf{P}^{(n,k)}(U,U_{1};(\cdot )) & = & \mathbf{P}^{(n)}(\left\{ \mathbf{Z%
}(un,n),0\leq u\leq U\right\} \in (\cdot )|\left\Vert \mathbf{Z}%
(U_{1}n,n)\right\Vert \leq k), \\
\mathbf{\bar{P}}^{(n,k)}(U,U_{1};(\cdot )) & = & \mathbf{P}^{(n)}(\left\{
\mathbf{Z}(un,n),0\leq u\leq U\right\} \in (\cdot )|\left\Vert \mathbf{Z}%
(U_{1}n,n)\right\Vert >k) \\
\mathcal{P}(U;(\cdot )) & = & \mathbf{P}(\left\{ \mathbf{W}(u),0\leq u\leq
U\right\} \in (\cdot )), \\
\mathcal{P}^{(k)}(U,U_{1};(\cdot )) & = & \mathbf{P}(\left\{ \mathbf{W}%
(u),0\leq u\leq U\right\} \in (\cdot )|\left\Vert \mathbf{W}%
(U_{1})\right\Vert \leq k).%
\end{array}%
\end{equation*}%
Then, for a continuous real function $\psi $ on $\mathbf{D}_{[0,U]}(\mathbf{Z}%
_{+}^{2})$ such that $\left\vert \psi \right\vert \leq q$ for a positive $q$
we have
\begin{eqnarray*}
\int \psi (x)\mathbf{P}^{(n)}(U;dx) &=&\mathbf{P}^{(n)}(\left\Vert \mathbf{Z}%
(U_{1}n,n)\right\Vert >k)\int \psi (x)\mathbf{\bar{P}}^{(n,k)}(U,U_{1};dx) \\
&&+\mathbf{P}^{(n)}(\left\Vert \mathbf{Z}(U_{1}n,n)\right\Vert \leq k)\int
\psi (x)\mathbf{P}^{(n,k)}(U,U_{1};dx).
\end{eqnarray*}%
For the first summand we get%
\begin{multline*}
\lim \sup_{n\rightarrow \infty }\mathbf{P}^{(n)}(\left\Vert \mathbf{Z}%
(U_{1}n,n)\right\Vert >k)\int \psi (x)\mathbf{\bar{P}}^{(n,k)}(U,U_{1};dx) \\
\leq q\lim \sup_{n\rightarrow \infty }\mathbf{P}^{(n)}(\left\Vert \mathbf{Z}%
(U_{1}n,n)\right\Vert >k)=q\mathbf{P}(\left\Vert \mathbf{W}%
(U_{1})\right\Vert >k)=o(1)
\end{multline*}%
as $k\rightarrow \infty $ by the properties of $\mathbf{W}(\cdot )$. On the
other hand, letting first $n\rightarrow \infty $ and then $k\rightarrow
\infty $ we obtain from Lemma \ref{L_Kskrokh} and Theorem \ref%
{T_skoroh},
\begin{multline*}
\lim_{k\rightarrow \infty }\lim_{n\rightarrow \infty }\mathbf{P}%
_{n}(\left\Vert \mathbf{Z}(U_{1}n,n)\right\Vert \leq k)\int \psi (x)\mathbf{P%
}^{(k)}(U,U_{1};dx) \\
=\lim_{k\rightarrow \infty }\mathbf{P}(0<\left\Vert \mathbf{W}%
(U_{1})\right\Vert \leq k)\int \psi (x)\mathcal{P}^{(k)}(U,U_{1};dx) \\
=\lim_{k\rightarrow \infty }\int_{\left\{ 0<\left\Vert \mathbf{W}%
(U_{1})\right\Vert \leq k\right\} }\psi (x)\mathcal{P}(U,U_{1};dx)=\int \psi
(x)\mathcal{P}(U;dx).
\end{multline*}%
Thus,%
\begin{equation*}
\lim_{n\rightarrow \infty }\int \psi (x)\mathbf{P}^{(n)}(\mathbf{Z}((\cdot
)n,n)\in dx)=\int \psi (x)\mathcal{P}(U;dx)
\end{equation*}%
for any bounded continuous function on $\mathbf{D}_{[0,U]}(\mathbf{Z}_{+}^{2})$
proving Theorem \ref{T_VatSagmicro}(1).\newline

\noindent \textbf{Proof of tightness for Theorem \ref{T_VatSagmicro} (2).} For $0<t<\infty $ and $1\ll h(n)\ll n$ let
\begin{equation*}
\mathbf{\hat{Z}}(n;t)=\left( \hat{Z}_{1}(n;t),\hat{Z}_{2}(n;t)\right)
=\left( \frac{Q_{21}(n)}{nQ_{1}(th(n))}Z_{1}(n-th(n),n),\frac{Q_{2}(n)}{%
Q_{2}(th(n))}Z_{2}(n-th(n),n)\right) .
\end{equation*}%
For $0<t_{1}<t_{0}$ we write%
\begin{equation*}
Z_{1}(n-t_{1}h(n),n)=Z_{1}^{(1)}(n-t_{1}h(n),n;t_{0})+Z_{1}^{(2)}(n-t_{1}h(n),n;t_{0})
\end{equation*}%
where $Z_{1}^{(i)}(n-t_{1}h(n),n;t_{0}),i=1,2$ is the number of
type $1$ particles in the reduced process at moment $n-t_{1}h(n)$ with type $%
i$ ancestors at time $n-t_{0}h(n)$. Set
\begin{equation*}
\mathbf{\hat{z}}_{n}=\left( \hat{z}_{n1},\hat{z}_{n2}\right) =
\left( \frac{Q_{21}(n)}{nQ_{1}(t_{0}h(n))}z_{1},\frac{Q_{2}(n)}{%
Q_{2}(t_{0}h(n))}z_{2}\right) ,\quad r_{ni}=\frac{%
Q_{i}(t_{1}h(n))}{Q_{i}(t_{0}h(n))},\quad i\in \{1,2\}
\end{equation*}%
and, for $\varepsilon >0$
\begin{equation*}
\Lambda _{n1}(t)=\varepsilon {nQ_{1}(th(n))}/{Q_{21}(n)},\quad \Lambda
_{n2}(t)=\varepsilon {Q_{2}(th(n))}/{Q_{2}(n)}.
\end{equation*}%
Note that under our conditions $\Lambda _{ni}(t)\rightarrow \infty $ as $%
n\rightarrow \infty $. With these notations we have 
\begin{multline*}
\left\{ \left\Vert \mathbf{\hat{Z}}(n;t_{1})-\mathbf{\hat{z}}_{n}\right\Vert
>3\varepsilon \right\} \subset \Big\{\left\vert
Z_{2}(n-t_{1}h(n),n)-r_{n2}z_{2}\right\vert >\Lambda _{n2}(t_{1})\Big\} \\
\cup \left\{ \left\vert
Z_{1}^{(1)}(n-t_{1}h(n),n;t_{0})-r_{n1}z_{1}\right\vert >\Lambda
_{n1}(t_{1})\right\} \cup \left\{ Z_{1}^{(2)}(n-t_{1}h(n),n;t_{0})>\Lambda
_{n1}(t_{1})\right\} .
\end{multline*}%
To go further we need the representation%
\begin{equation*}
\hat{Z}_{2}(n;t_{1})=\frac{Q_{2}(n)}{Q_{2}(t_{1}h(n))}%
\sum_{k=1}^{Z_{2}(n-t_{0}h(n),n)}\zeta _{k}(n;t_{0},t_{1}),
\end{equation*}%
where
$\zeta _{k}(n;t_{0},t_{1})\overset{d}{=}\zeta (n;t_{0},t_{1})
$ and $\zeta (n;t_{0},t_{1})$ is the number of type $2$ particles in the
reduced process at time $n-t_{1}h(n)$ being descendants of a particle of the
reduced process existing at time $n-t_{0}h(n)$. Observe that $\zeta
(n;t_{0},t_{1})\leq Z_{2}(n-t_{1}h(n))$ and $\mathbf{E}Z_{2}^{\gamma
}(n-t_{1}h(n))<\infty $ for any $\gamma \in \lbrack 1,1+\alpha _{2})$ in
view of Condition (\ref{F2}) and Theorem 2.3.7 in \cite{Sev72}. 
Notice that
\begin{eqnarray*}
\mathbf{E}\left[ \zeta (n;t_{0},t_{1})|Z_{2}(n-t_{0}h(n),n)=1\right]  &=&%
\mathbf{E}_{2}\left[ Z_{2}(\left( t_{0}-t_{1}\right) h(n),t_{0}h(n))|Z_{2}(t_{0}h(n))>0\right]  \\
&=&{Q_{2}(t_{1}h(n))}/{Q_{2}(t_{0}h(n))}
\end{eqnarray*}%
and that $\{\zeta _{k}(n;t_{0},t_{1}),\
k=1,2,...,Z_{2}(n-t_{0}h(n),n)\}$ are iid given $Z_{2}(n-t_{0}h(n),n)$.
Using von Bahr-Esseen inequality (\cite{BE65} Theorem 2) for $%
\gamma \in \left( 1,1+\alpha _{2}\right) ,$ Markov branching property and
Lemma \ref{L_convol}, we obtain that
\begin{multline*}
\left( {\Lambda _{n2}(t_{1})}\right) ^{\gamma }\mathbf{P}^{(n)}\left(
\,\left\vert Z_{2}(n-t_{1}h(n),n)-r_{n2}z_{2}\right\vert
>\Lambda _{n2}(t_{1})|\mathbf{\hat{Z}}(n;t_{0})=\mathbf{\hat{z}}_{n}\right)
\\
\leq \mathbf{E}\left[ \left\vert Z_{2}(n-t_{1}h(n),n)
-r_{n2}z_{2}\right\vert ^{\gamma }|\mathbf{\hat{Z}}(n;t_{0})=\mathbf{\hat{z}}%
_{n}\mathbf{,Z}(n)\neq \mathbf{0}\right]  \\
\leq 2\mathbf{E}\left[ \sum_{k=1}^{z_{2}}\left\vert \zeta
_{k}(n;t_{0},t_{1})-r_{n2}\right\vert ^{\gamma }|\mathbf{\hat{Z}}(n;t_{0})=%
\mathbf{\hat{z}}_{n}\mathbf{,Z}(n)\neq \mathbf{0}\right]  \\
=2z_{2}\mathbf{E}_{2}\left[ \left\vert Z_{2}(\left(
t_{0}-t_{1}\right) h(n),t_{0}h(n))-r_{n2}\right\vert ^{\gamma }|\mathbf{Z}%
(t_{0}h(n))\neq \mathbf{0}\right]  \\
\leq 2z_{2}\left( \frac{2^{\gamma }}{Q_{21}(t_{0}h(n))}\mathbf{E}%
_{2}\left[ Z_{2}^{\gamma }(\left( t_{0}-t_{1}\right) h(n),t_{0}h(n))\right]
+2^{\gamma }r_{n2}^{\gamma }\right) ,
\end{multline*}%
where in the last inequality we have applied, for $a,b\in \mathbf{R}_{2}$
\begin{equation}
|a-b|^{\gamma }\leq \max (|2a|^{\gamma },|2b|^{\gamma })\leq 2^{\gamma
}(|a|+|b|).  \label{ineq2gamma}
\end{equation}%
Further, let%
\begin{equation*}
I_{k}\left( t_{0},t_{1};n\right) ,\quad k=1,2,...,Z_{2}(\left(
t_{0}-t_{1}\right) h(n))
\end{equation*}%
be the indicator of the event that the $k$th particle belonging to
generation $\left( t_{0}-t_{1}\right) h(n)$ has descendants in generation $%
t_{0}h(n)$ and
$
I_{k}\left( t_{0},t_{1};n\right) \overset{d}{=}I\left( t_{0},t_{1};n\right) .
$
Now, again by \eqref{ineq2gamma} and the Bahr-Esseen inequality
\begin{multline*}
2^{-\gamma }\mathbf{E}_{2}\left[ \left\vert Z_{2}(\left( t_{0}-t_{1}\right)
h(n),t_{0}h(n))\right\vert ^{\gamma }\right]  \\
=2^{-\gamma }\mathbf{E}_{2}\left[ \left\vert \sum_{k=1}^{Z_{2}(\left(
t_{0}-t_{1}\right) h(n))}\left( I_{k}\left( t_{0},t_{1};n\right)
-Q_{2}(t_{1}h(n))\right) +Z_{2}(\left( t_{0}-t_{1}\right)
h(n))Q_{2}(t_{1}h(n))\right\vert ^{\gamma }\right]  \\
\leq \mathbf{E}_{2}\left[ \left\vert \sum_{k=1}^{Z_{2}(\left(
t_{0}-t_{1}\right) h(n))}\left( I_{k}\left( t_{0},t_{1};n\right)
-Q_{2}(t_{1}h(n))\right) \right\vert ^{\gamma }\right] +\mathbf{E}_{2}\left[
Z_{2}^{\gamma }(\left( t_{0}-t_{1}\right) h(n))Q_{2}^{\gamma }(t_{1}h(n))%
\right]  \\
\leq \mathbf{E}_{2}\left[ Z_{2}(\left( t_{0}-t_{1}\right) h(n))\right]
\mathbf{E}_{2}\left[ \left\vert \mathbf{1}_{\{Z_{2}(t_{1}h(n))>0%
\}}-Q_{2}(t_{1}h(n))\right\vert ^{\gamma }\right]  \\
+Q_{2}^{\gamma }(t_{1}h(n))\mathbf{E}_{2}\left[ Z_{2}^{\gamma }(\left(
t_{0}-t_{1}\right) h(n))\right] .
\end{multline*}%
Observe now that
\begin{multline*}
\mathbf{E}\left[ \left\vert \mathbf{1}_{\{Z_{2}(t_{1}h(n))>0%
\}}-Q_{2}(t_{1}h(n))\right\vert ^{\gamma }\right]  \\
=Q_{2}(t_{1}h(n))\left\vert 1-Q_{2}(t_{1}h(n))\right\vert ^{\gamma }+\left(
1-Q_{2}(t_{1}h(n))\right) Q_{2}^{\gamma }(t_{1}h(n))\leq CQ_{2}(t_{1}h(n))
\end{multline*}%
and by Lemma 11 in \cite{VVF}%
\begin{equation*}
\mathbf{E}_{2}\left[ Z_{2}^{\gamma }(\left( t_{0}-t_{1}\right) h(n))\right]
\leq CQ_{2}^{1-\gamma }(\left( t_{0}-t_{1}\right) h(n))\leq CQ_{2}^{1-\gamma
}(t_{0}h(n)).
\end{equation*}%
As a result we get%
\begin{equation*}
\mathbf{E}_{2}\left[ Z_{2}^{\gamma }(\left( t_{0}-t_{1}\right)
h(n),t_{0}h(n))\right] \leq CQ_{2}(t_{1}h(n)).
\end{equation*}%
This shows that, for each fixed pair $0<p_{1}<p_{2}<\infty $
\begin{equation*}
\mathbf{P}^{(n)}\left( \,\left\vert Z_{2}
(n-t_{1}h(n),n)-r_{n2}z_{2}\right\vert >\Lambda _{n2}(t_{1})|\mathbf{\hat{Z}%
}(n;t_{0})=\mathbf{\hat{z}}_{n}\right)\leq Cz_{2}\Big(\frac{1}{\Lambda
_{n}(t_{1})}\Big)^{\gamma }\leq C\Big(\frac{1}{\Lambda _{n}(p_{1})}\Big)%
^{\gamma -1}
\end{equation*}%
which goes to $0$ as $n\rightarrow \infty $ uniformly in $%
\left( t_{0},t_{1}\right) \in \left[ p_{1},p_{2}\right] $ and
\begin{equation*}
z_{2}\leq C\Lambda _{n2}(p_{1})\text{ }\ \text{or \ }\hat{z}%
_{2}=z_{2}/\Lambda _{n2}(t_{1})\leq C_{1}.
\end{equation*}
Further, for $\gamma \in (1,1+\alpha _{1})$ we have%
\begin{multline*}
\mathbf{P}^{(n)}\left( \left\vert Z_{1}^{(1)}(n-t_{1}h(n),n;t_{0})-r_{n1}z_{1}\right\vert >\Lambda _{n1}(t_{1})|\,%
\mathbf{\hat{Z}}(n;t_{0})=\mathbf{\hat{z}}_{n}\right)  \\
\\ \leq \left( \frac{1}{\Lambda _{n1}(t_{1})}\right) ^{\gamma }\mathbf{E}_{2}%
\left[ \left\vert Z_{1}^{(1)}(n-t_{1}h(n),n;t_{0})
-r_{n1}z_{1}\right\vert ^{\gamma }|\,\mathbf{\hat{Z}}(n;t_{0})=\mathbf{\hat{z%
}}_{n}\mathbf{,}\,\mathbf{Z}(n)\neq \mathbf{0}\right] .
\end{multline*}%
We now use the representation%
\begin{equation*}
Z_{1}^{(1)}(n-t_{1}h(n),n;t_{0})
=\sum_{k=1}^{Z_{1}(n-t_{0}h(n),n)}\eta _{k}(n;t_{0},t_{1})
\end{equation*}%
where
$\eta _{1k}(n;t_{0},t_{1})\overset{d}{=}\eta (n;t_{0},t_{1})
$ and $\eta (n;t_{0},t_{1})$ is the number of type $1$ particles in the
reduced process at time $n-t_{1}h(n)$ being descendants of a type $1$
particle existing in the reduced process at time $n-t_{0}h(n)$. Observe that
\begin{equation*}
\mathbf{E}\left[ \eta ^{\gamma }(n;t_{0},t_{1})\right] \leq \mathbf{E}_{1}%
\left[ Z_{1}^{\gamma }(n-t_{1}h(n))\right] <\infty
\end{equation*}%
for any $\gamma \in \lbrack 1,1+\alpha _{1})$ in view of
Condition (\ref{F1}) and Theorem 2.3.7 in \cite{Sev72}. Similarly to the previous estimates%
\begin{multline*}
\mathbf{E}\left[ \sum_{k=1}^{Z_{1}(n-t_{0}h(n),n)}\left\vert \eta
_{k}(n;t_{0},t_{1})-r_{n1}\right\vert ^{\gamma }|\mathbf{\hat{Z}}(n;t_{0})=%
\mathbf{\hat{z}}_{n},\mathbf{Z}(n)\neq \mathbf{0}\right]  \\
\leq 2z_{1}\mathbf{E}_{1}\left[ \left\vert \eta (n;t_{0},t_{1})-%
\frac{Q_{1}(t_{1}h(n))}{Q_{1}(t_{0}h(n))}\right\vert ^{\gamma
}|Z_{1}(n-t_{0}h(n),n)=1\right]  \\
=2z_{1}\mathbf{E}_{1}\left[ \left\vert Z_{1}(\left(
t_{0}-t_{1}\right) h(n),t_{0}h(n))-\frac{Q_{1}(t_{1}h(n))}{Q_{1}(t_{0}h(n))}%
\right\vert ^{\gamma }|Z_{1}(t_{0}h(n))>0\right] \leq Cz_{1}.
\end{multline*}%
Hence, for each fixed pair $0<p_{1}<p_{2}<\infty $ we get%
\begin{equation*}
\mathbf{P}^{(n)}\left( \left\vert Z_{1}^{(1)
}(n-t_{1}h(n),n;t_{0})-r_{n1}z_{1}\right\vert >\Lambda _{n1}(t_{1})|\mathbf{%
Z}(n;t_{0})=\mathbf{z}\right) \leq \left( \frac{1}{\Lambda _{n1}(t_{1})}%
\right) ^{\gamma }z_{1}\rightarrow 0
\end{equation*}%
as $n\rightarrow \infty $ uniformly in $\left( t_{0},t_{1}\right) \in \left[
\mathbf{p}_{1}\mathbf{,p}_{2}\right] $ and $z_{1}\leq C\Lambda _{n1}^{\gamma
}(p_{1})$ or $\hat{z}_{1}=z_{1}/\Lambda _{n1}^{\gamma }(t_{1})\leq C_{1}.$

Finally, using the branching property, we get
\begin{multline*}
\Lambda _{n1}(t_{1})\mathbf{P}^{(n)}\left( Z_{1}^{(2)%
}(n-t_{1}h(n),n;t_{0})>\Lambda _{n1}(t_{1})|\,\mathbf{\hat{Z}}(n;t_{0})=%
\mathbf{\hat{z}}_{n}\right)  \\
\leq \mathbf{E}\left[ Z_{1}^{(2)}(n-t_{1}h(n),n;t_{0})|\,%
\mathbf{\hat{Z}}(n;t_{0})=\mathbf{\hat{z}}_{n}\mathbf{,Z}(n)\neq \mathbf{0}%
\right]  \\
=z_{2}\mathbf{E}_{2}\left[ Z_{1}(\left( t_{0}-t_{1}\right)
h(n);t_{0}h(n))|\,\mathbf{Z}(t_{0}h(n))\neq \mathbf{0}\right] \\
=z_{2}\frac{A_{21}\left( t_{0}-t_{1}\right) h(n)Q_{1}(t_{1}h(n))%
}{Q_{21}(t_{0}h(n))}\leq C\frac{z_{2}}{\Lambda _{n1}(t_{1})}.
\end{multline*}%
Thus, for each fixed pair $0<p_{1}<p_{2}<\infty $%
\begin{equation*}
\mathbf{P}^{(n)}\left( Z_{1}^{(2)}%
(n-t_{1}h(n),n;t_{0})>\Lambda _{n1}(t_{1})|\mathbf{\hat{Z}}(n;t_{0})=%
\mathbf{\hat{z}}_{n}\right) \leq C\frac{z_{2}}{\Lambda _{n1}^{2}(t_{1})}%
\rightarrow 0
\end{equation*}%
as $n\rightarrow \infty $ uniformly in $\left( t_{0},t_{1}\right) \in \left[
p_{1},p_{2}\right] $ and $\hat{z}_{2}=z_{2}/\Lambda _{n1}(t_{1})\leq C_{1}.$

The remaining part of the proof follows the line of proving Theorem \ref%
{T_VatSagmicro}(1) and is based on the one-dimensional convergence
established in Theorem \ref{T_VatSag}(2). We omit the details.

\section{Most recent common ancestor\label{S_MRCA}}

We present in this section the derivation of the laws of the
MRCA. They are direct consequences of Theorems \ref{T_ZubQ1micro}-\ref%
{T_VatSagmicro}.

\noindent\textbf{Point (1)}: It follows from Theorem \ref{T_ZubQ1micro} that, given
Condition (\ref{Q1NeglQ2}) particles of type $2$ are always
present in the limiting process. Thus, the MRCA should be of type $2$. \
Since, given $\mathbf{Z}(n)\neq \mathbf{0}$ the event $\left\{
Z_{2}(m,n)=1\right\} $ has a nonnegligible probability only if $%
\lim_{n\rightarrow \infty }mn^{-1}<1,$ to find the distribution of the MRCA
in this situation it is sufficient to evaluate the derivative with respect
to $s_{2}$ of the limiting probability generating function (\ref{Zub12}) at
point $s_{2}=0$. This gives, as desired
\begin{equation*}
\lim_{n\rightarrow \infty }\mathbf{P}_{2}\left( Z_{2}\left( an,n\right) =1|%
\mathbf{Z}(n)\neq \mathbf{0}\right) =\lim_{n\rightarrow \infty }\mathbf{P}%
_{2}\left( \beta _{n}\geq an|\mathbf{Z}(n)\neq \mathbf{0}\right) =1-a.
\end{equation*}

\noindent \textbf{Point (2)}: According to the offspring generating
function of the limit process, given in \eqref{pgfeq1}, the law of the
offspring vector $(\xi _{21},\xi _{22})$ produced by a type $2$ individual satisfies:
\begin{equation}
\mathbf{P}(\xi _{21}=0,\xi _{22}\geq 2)=\frac{\alpha _{2}}{1+\alpha _{2}}%
\quad \text{and}\quad \mathbf{P}(\xi _{21}=1,\xi _{22}=0)=\frac{1}{1+\alpha
_{2}}.  \label{prop1Y}
\end{equation}%
Hence there are two possibilities:

\begin{enumerate}
\item[$\bullet $] Either the initial type $2$ individual generates several
type $2$ individuals at the firth birth event. Then the time of the
branching corresponds to the death of the MRCA.

\item[$\bullet $] Or the initial type 2 individual gives birth to a type 1
individual at the first birth event. Then the process does not evolve
anymore as the type 1 individuals are immortal and sterile. We
will see that in this case the MRCA birth time is of order
larger than $g^{\ast }(n)$.
\end{enumerate}

From \eqref{prop1Y} we get that
\begin{equation*}
\mathbf{P}(Y_{1}(t)+Y_{2}(t)=1\quad \text{for any }t>0)=\frac{1}{1+\alpha
_{2}}.
\end{equation*}%
Moreover, from points (2)-(3) of Theorem \ref{T_zubQ3}, we see that
for $g^{\ast }(n)\ll m\ll n$
\begin{equation*}
\lim_{n\rightarrow \infty }\mathbf{P}_{2}(Z_{1}(m,n)=1|\mathbf{Z}(n)\neq 0)=%
\frac{1}{1+\alpha _{2}},
\end{equation*}%
and for $a\in (0,1)$
\begin{equation*}
\lim_{n\rightarrow \infty }\mathbf{P}_{2}(Z_{1}(an,n)=1|\mathbf{Z}(n)\neq 0)=%
\frac{1-a}{1+\alpha _{2}}.
\end{equation*}%
As the process $Z_{1}(m,n)$ is non-decreasing with $m$, we
deduce that for  $g^{\ast }(n)\ll m\ll n$ and $a\in (0,1)$,
\begin{equation*}
\lim_{n\rightarrow \infty }\mathbf{P}_{2}(m\ll \beta _{n}\leq an|\mathbf{Z}%
(n)\neq 0)=\frac{a}{1+\alpha _{2}}.
\end{equation*}%
If, on the contrary the initial type 2 individual gives birth to several
type 2 individuals, which happens with a probability $\alpha
_{2}/(1+\alpha _{2})$, the time of this first branching will be an
exponential variable with parameter $(1+\alpha _{2})/\alpha _{2}$.
Converting this distribution in terms of the limit law for the MRCA yields
\begin{equation*}
\lim_{n\rightarrow \infty }\mathbf{P}_{2}(\beta _{n}\leq tg^{\ast }(n)|%
\mathbf{Z}(n)\neq 0)=\frac{\alpha _{2}}{1+\alpha _{2}}(1-e^{-t(1+\alpha
_{2})/\alpha _{2}}).
\end{equation*}%
This ends the proof of Theorem \ref{T_MRCA}(2).

\noindent\textbf{Point (3)}: To find the distribution of the MRCA birth time under
Condition (\ref{VaSag}) \ we use the offspring generating functions of the
process $\left\{ \mathbf{W}(t),t\geq 0\right\} $, which are given in~%
\eqref{pgfeq3} and \eqref{pgfeqW}. Equation \eqref{pgfeqW} implies the
following properties for the law of the offspring vector $(\xi _{21},\xi
_{22})$ produced by a type $2$ individual of $\mathbf{W(\cdot )}$:
\begin{equation}
\mathbf{P}_{2}(\xi _{21}=1,\xi _{22}=0)=\frac{\sigma \alpha _{2}A_{21}}{%
b((1+\alpha _{2})b^{\alpha _{2}}-1)},\  \mathbf{P}_{2}(\xi _{21}=0,\xi
_{22}\geq 2)=\frac{\alpha _{2}b^{\alpha _{2}}}{(1+\alpha _{2})b^{\alpha
_{2}}-1}.  \label{prop1W}
\end{equation}%
Moreover, we see from the definition of $W_{1}(\cdot )$ and \eqref{pgfeq3}
that a type $1$ individual dies after an exponentially distributed time with
parameter $1$ and gives birth to two type $1$ individuals at least. Hence if
we denote by $T$ the time of the first branching there are two possibilities:

\begin{enumerate}
\item[$\bullet $] Either the initial type $2$ individual gives birth to
several type $2$ individuals. In this case $T$ will be an exponential
variable with parameter $\kappa$, where $\kappa$ has been
defined in \eqref{defkappa}. According to \eqref{prop1W} this happens with
probability
\begin{equation*}
\frac{\alpha _{2}b^{\alpha _{2}}}{(1+\alpha _{2})b^{\alpha _{2}}-1} = \frac{b^{\alpha_{2}}}{\kappa} .
\end{equation*}

\item[$\bullet $] Or the initial type $2$ individual gives birth to one type
$1$ individual. Then this type $1$ individual has an exponential distributed
life-time with parameter $1$. In this case, $T$ will be the sum of an
exponential variable with parameter $\kappa$ and of an
independent exponential variable with parameter $1$. According to %
\eqref{prop1W} this happens with probability
\begin{equation*}
\frac{\sigma \alpha _{2}A_{12}}{b((1+\alpha _{2})b^{\alpha _{2}}-1)}= \frac{\sigma A_{21}}{b\kappa} .
\end{equation*}
\end{enumerate}

Hence we get
\begin{equation*}
\mathbf{P}(T\leq t)=\frac{b^{\alpha _{2}}}{\kappa }(1-e^{-\kappa t})+\frac{%
\sigma A_{12}}{b\kappa }\Big(1+\frac{1}{\kappa -1}e^{-\kappa t}-\frac{\kappa
}{\kappa -1}e^{-t}\Big).
\end{equation*}%
By using several times the equality $b^{1+\alpha _{2}}-b=\sigma \alpha
_{2}A_{12}$ given in \eqref{defb}, we finally obtain
\begin{equation*}
\mathbf{P}(T\leq t)=1-\frac{1}{1+\alpha _{2}}e^{-t}-\frac{\alpha _{2}}{%
1+\alpha _{2}}e^{-t((1+\alpha _{2})b^{\alpha _{2}}-1)/\alpha _{2}}.
\end{equation*}%
Then by applying Theorem \ref{T_VatSagmicro}(1), we get for
every $a\in (0,1)$:
\begin{equation*}
\lim_{n\rightarrow \infty }\mathbf{P}(\beta _{n}\leq an)=\mathbf{P}(T\leq
-\ln (1-a)).
\end{equation*}%
This ends the proof of Theorem \ref{T_MRCA}(3).


\end{document}